\documentclass[onefignum,onetabnum]{siamart}

\usepackage{amsfonts}
\usepackage{graphicx}
\usepackage{epstopdf}
\usepackage{algorithmic}
\ifpdf
  \DeclareGraphicsExtensions{.eps,.pdf,.png,.jpg}
\else
  \DeclareGraphicsExtensions{.eps}
\fi

\newsiamremark{remark}{Remark}
\newsiamremark{hypothesis}{Hypothesis}
\crefname{hypothesis}{Hypothesis}{Hypotheses}
\newsiamthm{claim}{Claim}

\headers{Lagrange multiplier schemes for Hamiltonian PDEs}{Y.H. Bo and Y.S. Wang}

\title{High-order Lagrange multiplier schemes for general Hamiltonian PDEs\thanks{\textbf{Funding:} This work is supported by the National Natural Science Foundation of China (Grant No. 12526576, 12171245), the University Annual Scientific Research Plan of Anhui Province (Grant No. 2022AH050200), and the Foundation of Anhui Normal University (Grant No. 903762201).}}

\author{
Yonghui Bo\thanks{Corresponding author. School of Mathematics and Statistics, Anhui Normal University, Wuhu, 241002, China (\email{boyonghui@ahnu.edu.cn}).}
\and Yushun Wang\thanks{Ministry of Education Key Laboratory for NSLSCS, Jiangsu Collaborative Innovation Center of Biomedical Functional Materials, School of Mathematical Sciences, Nanjing Normal University, Nanjing, 210023, China (\email{wangyushun@njnu.edu.cn}).}
}

\begin{document}
\maketitle
\begin{abstract}
In this paper, we introduce a Lagrange multiplier approach to construct linearly implicit energy-preserving schemes of arbitrary order for general Hamiltonian PDEs. Unlike the widely used auxiliary variable methods, this novel approach does not require the nonlinear part of the energy to be bounded from below, thereby offering broader applicability. Moreover, this approach preserves the original energy exactly at both the continuous and discrete levels, as opposed to a modified energy preserved by the auxiliary variable methods. Rigorous proofs are provided for the energy conservation and numerical accuracy of all derived schemes. The trade-off for these advantages is the need to solve a nonlinear algebraic equation to determine the Lagrange multiplier. Nevertheless, numerical experiments show that the associated computational cost is generally not dominant, indicating that the new schemes retain computational efficiency comparable to the auxiliary variable-based schemes. Numerical results demonstrate the efficiency, accuracy, and structure-preserving properties of the proposed schemes.
\end{abstract}



\section{Introduction}\label{sec:introduction}
Many important conservation models, including the nonlinear Schr\"odinger equation, the sine-Gordon equation, and the Korteweg-de Vries equation, can be reformulated as Hamiltonian systems \cite{leimkuhler-04-SHD}. In this paper, we consider Hamiltonian PDEs of the form
\begin{equation}\label{eq-1-1}
	z_t=\mathcal{S}\frac{\delta\mathcal{H}}{\delta z},
\end{equation}
where $\mathcal{S}$ is a skew-adjoint operator independent of the solution variable $z(x,t)$, with $(x,t)\in\Omega\times [0,T]$ and $\Omega\subset\mathbb{R}^d$. A most intuitive property of \eqref{eq-1-1} is the conservation of energy, meaning that the Hamiltonian $\mathcal{H}$ is invariant under the continuous flow, i.e.,
\begin{equation}\label{eq-1-2}
	\frac{d}{dt}\mathcal{H}(z)=0,\quad\mathcal{H}(z)=\frac{1}{2}\big(z,\mathcal{L}z\big) +\big(N(z),1\big),
\end{equation}
where $\frac{\delta\mathcal{H}}{\delta z}=\mathcal{L}z+N'(z)$, $\mathcal{L}$ is a symmetric non-negative linear operator, $N(z)$ denotes the nonlinear energy density, and $(\cdot, \cdot)$ represents the standard $L^2$ inner product on $\Omega$. Nowadays, numerical schemes that preserve a discrete analogue of \eqref{eq-1-2} are often preferred over non-conservative methods, especially for long-time simulations.

Fully implicit schemes have served as a cornerstone in the development of energy-preserving methods and continue to be of fundamental importance. For problems with quadratic energy functionals, symplectic Runge-Kutta (RK) methods \cite{hairer-06-GNI-ODE} are naturally energy-preserving and can achieve arbitrary accuracy by increasing the number of stages. For more general forms of energy, discrete gradient methods \cite{mcLachlan-99-DG-PTRSLA} provide a systematic framework for constructing energy-preserving schemes. Among these, the averaged vector field (AVF) method \cite{quispel-08-AVF-JPAMT} and its extensions \cite{cai-18-PAVF-JCP,wu-13-oscil-JCP} are particularly prominent. These methods were later generalized within the discrete variational derivative (DVD) framework \cite{furihata-11-DVD-CHCRC}. Driven by the growing demand for high-precision simulations, high-order methods have attracted significant attentions. Examples include high-order AVF schemes \cite{hairer-10-coll-JNAIAM,Li-16-AVF-JCM}, time finite element methods \cite{betsch-00-time-FEM-JCP,tang-12-time-FEM-AMC}, and Hamiltonian boundary value methods (HBVMs) \cite{luigi-19-HBVM-KDV-JCAM,luigi-10-HBVM-JNAIAM,luigi-20-gyrocenter-HBVM-JCAM}. In particular, HBVMs have been employed as spectral methods for efficiently solving some classes of Hamiltonian PDEs \cite{luigi-19-spectral-space-time-HPDEs-NA}. Despite these advances, all aforementioned methods are fully implicit and require solving nonlinear systems. This often results in high computational costs, especially in large-scale simulations or for problems with strong nonlinearities. Additional challenges arise from the need for efficient nonlinear solvers \cite{luigi-11-efficient-HBVM-JCAM}. These aspects can limit their practicality in computationally intensive settings, which motivates ongoing research into more efficient alternatives.

Linearly implicit energy-preserving schemes provide an effective alternative to the above problems. These schemes not only preserve the energy but also enable the efficient computation by requiring only solving linear systems at each time step. For systems with polynomial energies, the multiple DVD method \cite{matsuo-01-DC-JCP} and the polarization approach \cite{dahlby-11-general-IP-SIAMJSC} have been proposed. However, no general methodology is available for conserving arbitrary forms of energy until the advent of auxiliary variable methods, which include the scalar auxiliary variable (SAV) \cite{shen-18-SAV-JCP} and the invariant energy quadratization (IEQ) \cite{yang-17-IEQ-JCP} methods. Although originally developed for phase field models, they have been successfully extended to various conservative systems, including Hamiltonian systems \eqref{eq-1-1} \cite{bo-22-EIEQ-MS,cai-20-linear-MS-JCP,hou-24-CF-CNLW-ANM,jiang-19-SG-IEQ-JSC,jiang-20-CH-SAV-JSC,zhang-20-RK-JCP}. These two methods share similar ideas in constructing linear schemes. To facilitate a comparative study with the Lagrange multiplier approach, we present a sketch of the SAV method. Let $F(z)=\big(N(z),1\big)$, the SAV method introduces an additional variable $r(t)=\sqrt{F(z)+c_0}$, under the restriction that $F(z)$ is bounded from below. The constant $c_0$ is selected to ensure that $F(z)+c_0>0$. Then, the system \eqref{eq-1-1} can be reformulated as
\begin{align}\label{eq-1-3}
	\aligned
	z_t&=\mathcal{S}\big(\mathcal{L}z+\frac{N'(z)}{\sqrt{F(z)+c_0}}r\big)\\
	r_t&=\frac{1}{2}\big(\frac{N'(z)}{\sqrt{F(z)+c_0}},z_t\big).
	\endaligned
\end{align}
The system \eqref{eq-1-3} satisfies a modified energy conservation law
\begin{equation}\label{eq-1-4}
	\frac{d}{dt}~\Big(\frac{1}{2}\big(z,\mathcal{L}z\big)+r^2\Big)=0.
\end{equation} 
The SAV and IEQ methods transform the original energy \eqref{eq-1-2} into a simple quadratic form, thereby simplifying the treatment of the nonlinear term. Then, a second-order linearly implicit energy-preserving scheme reads
\begin{equation}\label{eq-1-5}
	\begin{aligned}
			\frac{z^{n+1}-z^n}{\Delta t}&=\mathcal{S}\big(\mathcal{L}z^{n+\frac{1}{2}}+\frac{N'(\tilde{z}^{n+\frac{1}{2}})}{\sqrt{F(\tilde{z}^{n+\frac{1}{2}})+c_0}}r^{n+\frac{1}{2}}\big),\\
			\frac{r^{n+1}-r^n}{\Delta t}&=\frac{1}{2}\big(\frac{N'(\tilde{z}^{n+\frac{1}{2}})}{\sqrt{F(\tilde{z}^{n+\frac{1}{2}})+c_0}},\frac{z^{n+1}-z^n}{\Delta t}\big),
		\end{aligned}
\end{equation}
where $\Delta t$ is the time step, $z^{n+\frac{1}{2}}=\frac{z^n+z^{n+1}}{2}$ and $\tilde{z}^{n+\frac{1}{2}}=\frac{3z^n-z^{n-1}}{2}$. The scheme \eqref{eq-1-5} is named SAV-CN. It is not difficult to prove that SAV-CN is energy-preserving in the sense that 
\begin{equation}\label{eq-1-6}
	\mathcal{H}^{n+1}=\mathcal{H}^n,\quad\mathcal{H}^n=\frac{1}{2}\big(z^n,\mathcal{L}z^n\big)+\big(r^n\big)^2.
\end{equation}
Note that $r^2\neq F(z)$ when $c_0\neq 0$, the modified energy \eqref{eq-1-4} differs from the original energy \eqref{eq-1-2}. Consequently, the scheme \eqref{eq-1-5} does not preserve the original energy. Moreover, implementing \eqref{eq-1-5} numerically requires solving a linear system with variable coefficients, which in turn relies on nonlinear iterations and increases computational costs. Higher-order linear schemes \cite{akrivis-19-RK-SAV-SIAMJSC,bo-22-ESAV-NA,feng-21-SAV-Gauss-SINA,li-20-linear-high-order-SAV-NLW-JSC,li-21-high-order-SAV-NLS-JSC} can be constructed using symplectic RK methods and high-order extrapolation formulas. Although both the SAV and IEQ methods lead to linear schemes, the calculation of the solution variable and the auxiliary variable in the resulting schemes can not be decoupled. Then, extra nonlinear iterations must be applied previously to obtain the solution variable, which would become more complicated for high-order schemes. 

Recently, supplementary variable methods \cite{gong-21-SVM-CMAME} have been developed to construct linear schemes for general energies by adding extra terms to the original vector field to enforce the conservation of energy. In this paper, inspired by \cite{cheng-20-LM-CMAME}, we adopt an alternative approach wherein the Lagrange multiplier is multiplied directly to the nonlinear term without introducing extra terms, resulting in a modified system that is structurally simpler. The remarkable features are given as
\begin{itemize}
	\item[{i.}] The Lagrange multiplier is not based on any restrictions, so it is applicable to general Hamiltonian PDEs.
	\item[{ii.}] In both continuous and discrete cases, the new approach conserves the original energy \eqref{eq-1-2}, as opposed to a modified energy in the auxiliary variable methods.
	\item[{iii.}] We present a novel class of energy-preserving schemes that are computationally efficient and capable of arbitrary order accuracy. Although solving for the Lagrange multiplier is necessary, its computational cost is shown to be negligible.
	\item[{iv.}] The methodology of this paper can be simply generalized to other conservative or dissipative systems.
\end{itemize}

The remainder of this paper is structured as follows. Section \ref{sec:main} details the main results, starting with the Lagrange multiplier approach and a simple second-order linear scheme. Then, we develop high-order linearly implicit energy-preserving schemes via a prediction-correction strategy. This is followed by rigorous analyses of the original energy conservation and numerical accuracy for all proposed schemes. In Section \ref{sec:experiments}, numerical tests are provided to illustrate the effectiveness, convergence, and structure-preserving properties of the proposed schemes. Some concluding remarks are given in Section \ref{sec:conclusions}. 

\section{Main results}\label{sec:main}
We introduce below a novel Lagrange multiplier approach for Hamiltonian systems \eqref{eq-1-1}. The multiplier is multiplied directly on the nonlinear term of the original vector field without adding additional variables and terms.
\subsection{The Lagrange multiplier approach and its second-order scheme} 
Defining a scalar function $\lambda(t)$, the system \eqref{eq-1-1} can be rewritten as
\begin{align}\label{eq-2-1}
	\aligned
	z_t&=\mathcal{S}\big(\mathcal{L}z+\lambda(t)N'(z)\big),\\
	\frac{d}{dt}F(z)&=\lambda(t)\big(N'(z),z_t\big),
	\endaligned
\end{align}
where $F(z)=\big(N(z),1\big)$ is the nonlinear part of \eqref{eq-1-2}. If we set the initial condition $\lambda(0)=1$, then the new system \eqref{eq-2-1} is equivalent to \eqref{eq-1-1}, i.e., $\lambda(t)\equiv 1$ in \eqref{eq-2-1}. 
\begin{theorem}\label{Th-2-1}
The system \eqref{eq-2-1} satisfies the original energy conservation law
\begin{equation}\label{eq-2-2}
	\frac{d}{dt}~\Big(\frac{1}{2}\big(z,\mathcal{L}z\big)+\big(\mathcal{N}(z),1\big)\Big)=0.
\end{equation} 
\end{theorem}
\begin{proof}
Taking the inner product of the first equation in \eqref{eq-2-1} with $\mathcal{L}z+\lambda(t) N'(z)$ and summing it with the second equation, we obtain \eqref{eq-2-2}.
\end{proof}

Note that unlike the SAV method \eqref{eq-1-3} where $r(t)$ is just an auxiliary variable, $\lambda(t)$ here is to serve as a Lagrange multiplier to enforce conservation of the original energy. Furthermore, $F(z)$ does not require to be bounded from below. Similar to the SAV method, we can construct efficient numerical schemes for \eqref{eq-2-1}. For example, a novel second-order scheme based on the Crank-Nicolson method is as follows:
\begin{algorithm}[H]
\begin{align}\label{eq-2-3}
	\aligned
	\frac{z^{n+1}-z^n}{\Delta t}&=\mathcal{S}\big(\mathcal{L}z^{n+\frac{1}{2}}+\lambda^{n+\frac{1}{2}}N'(\tilde{z}^{n+\frac{1}{2}})\big),\\
	\frac{F(z^{n+1})-F(z^n)}{\Delta t}&=\lambda^{n+\frac{1}{2}}\big(N'(\tilde{z}^{n+\frac{1}{2}}),\frac{z^{n+1}-z^n}{\Delta t}\big).
	\endaligned
\end{align}
\end{algorithm}
\noindent Hereafter, the scheme \eqref{eq-2-3} is named LM-CN. There is the following energy conservation.
\begin{theorem}\label{Th-2-2}
The scheme LM-CN \eqref{eq-2-3} satisfies the semi-discrete original energy conservation law
\begin{equation}\label{eq-2-4}
	\mathcal{H}^{n+1}=\mathcal{H}^n,\quad  \mathcal{H}^n=\frac{1}{2}\big(z^n,\mathcal{L}z^n\big)+\big(N(z^n),1\big).
\end{equation} 
\end{theorem}
\begin{proof}
Computing the inner product of the first equation in \eqref{eq-2-3} with $\mathcal{L}z^{n+\frac{1}{2}}+\lambda^{n+\frac{1}{2}}N'(\tilde{z}^{n+\frac{1}{2}})$ and summing it with the second equation yields the preservation of original energy \eqref{eq-2-4}.
\end{proof}

We present an efficient method for solving the scheme LM-CN \eqref{eq-2-3}, beginning with a derivation of
\begin{equation}\label{eq-2-5}
	z^{n+1}=p^{n}+\lambda^{n+\frac{1}{2}}q^n,
\end{equation} 
where
\begin{equation}\label{eq-2-6}
	p^n=(I-\frac{1}{2}\Delta t\mathcal{S}\mathcal{L})^{-1}(I+\frac{1}{2}\Delta t\mathcal{S}\mathcal{L})z^n,\ q^n=\Delta t(I-\frac{1}{2}\Delta t\mathcal{S}\mathcal{L})^{-1}\mathcal{S}N'(\tilde{z}^{n+\frac{1}{2}}).
\end{equation}
The values of $p^n$ and $q^n$ can be computed explicitly. Substituting \eqref{eq-2-5} into the second equation of \eqref{eq-2-3}, we obtain
\begin{equation}\label{eq-2-7}
	\big(N(p^{n}+\lambda^{n+\frac{1}{2}}q^n)-N(z^n),1\big)=\lambda^{n+\frac{1}{2}}\big(N'(\tilde{z}^{n+\frac{1}{2}}),p^{n}+\lambda^{n+\frac{1}{2}}q^n-z^n\big).
\end{equation}  
As the equation \eqref{eq-2-7} is a nonlinear algebraic equation, it generally admits multiple solutions. However, since $\lambda^{n+\frac{1}{2}}$ approximates $1$, we can compute it efficiently using the Newton iteration with $1$ as the initial value. To summarize, the scheme LM-CN \eqref{eq-2-3} consists of the following steps:
\begin{algorithm}[H]
\begin{itemize}
	\item[{\textbf{$\star$}}]\textbf{step 1:} compute $p^n$ and $q^n$ from \eqref{eq-2-6};
	\item[{\textbf{$\star$}}]\textbf{step 2:} search $\lambda^{n+\frac{1}{2}}$ by solving \eqref{eq-2-7};
	\item[{\textbf{$\star$}}]\textbf{step 3:} obtain $z^{n+1}$ by \eqref{eq-2-5} and then proceed to the next step.
\end{itemize}
\end{algorithm}

The notation $(\cdot,\cdot)$ in \eqref{eq-2-7} denotes the $L^2$ inner product. Following spatial discretization, this is evaluated using the discrete $L^2$ inner product \cite{shen-11-spectral}. Note that only the linear constant-coefficient equations need to be solved, along with the algebraic equation \eqref{eq-2-7}. Since the computational cost of determining the Lagrange multiplier is negligible compared to solving \eqref{eq-2-6}, the scheme LM-CN \eqref{eq-2-3} is computationally efficient. Moreover, when combined with Fourier pseudo-spectral spatial discretization, the computations in \eqref{eq-2-6} can be significantly accelerated using fast Fourier transforms.

\subsection{High-order linearly implicit energy-preserving schemes}
This subsection shows how to easily derive higher-order linear schemes for solving \eqref{eq-2-1}. The proposed method utilizes the Gauss collocation methods integrated with a prediction-correction framework. For this purpose, we first introduce the following high-order prediction-correction scheme.
\subsubsection{The high-order prediction-correction scheme}
To facilitate the presentation, we recast \eqref{eq-1-1} in the following form
\begin{equation}\label{eq-2-8}
	z_t=f_1(z)+f_2(z),
\end{equation}
where $f_1(z)=\mathcal{S}\mathcal{L}z$ and $f_2(z)=\mathcal{S}N'(z)$. Applying the $s$-stage RK method \cite{hairer-06-GNI-ODE} with coefficients $a_{ij}$, $b_i$, $c_i$ ($i,j=1,2,\cdots,s$) to \eqref{eq-2-8} on the time interval $[t_n,t_{n+1}]$ leads to the following scheme
\begin{equation}\label{eq-2-9}
	\begin{aligned}
		z_i^n&=z^n+\Delta t\sum_{j=1}^sa_{ij}k_j^n,\\
		k_i^n&=f_1\big(z_i^n\big)+f_2\big(z_i^n\big),\\
		z^{n+1}&=z^n+\Delta t\sum_{i=1}^sb_ik_i^n.
	\end{aligned}
\end{equation}

Here, we allow a full matrix $(a_{ij})_{i,j=1}^s$ of non-zero coefficients. It should be noted, however, that solving such a fully implicit scheme can be computationally challenging. A nonlinear iterative solver must be employed, which entails substantial computational expenses. In particular, HBVM($s$,$s$), which is known to be equivalent to an $s$-stage Gauss method, requires the design of efficient nonlinear solvers, especially for high-dimensional problems \cite{luigi-11-efficient-HBVM-JCAM}. To mitigate this, we present the following high-order prediction-correction scheme: 
\begin{itemize}
	\item[{\textbf{$\star$}}]\textbf{step1 (prediction):} set $z_{i,(0)}^n=z^n$ and compute $k_{i,(m+1)}^n$ from $m=0$ to $\Lambda-1$ by solving
	\begin{equation}\label{eq-2-10}
	\begin{aligned}
		z_{i,(m+1)}^n&=z^n+\Delta t\sum_{j=1}^sa_{ij}k_{j,(m+1)}^n,\\ k_{i,(m+1)}^n&=f_1\big(z_{i,(m+1)}^n\big)+f_2\big(z_{i,(m)}^n\big),
		\end{aligned}
	\end{equation}
	\item[{\textbf{$\star$}}]\textbf{step 2 (correction):} update $z_{(\Lambda)}^{n+1}$ by
	\begin{equation}\label{eq-2-11}
		z_{(\Lambda)}^{n+1}=z^n+\Delta t\sum_{i=1}^sb_ik_{i,(\Lambda)}^n.
	\end{equation}
\end{itemize}
\noindent The following lemma holds for the two RK schemes described above.
\begin{lemma}\label{Le-2-3}
Assuming that $z^{n+1}$ and $z_{(\Lambda)}^{n+1}$ are solutions to the schemes \eqref{eq-2-9} and \eqref{eq-2-11}, respectively, it follows that
\begin{equation}\label{eq-2-12}
	z_{(\Lambda)}^{n+1}-z^{n+1}=\mathcal{O}\big((\Delta t)^{\Lambda+1}\big).
\end{equation} 
\end{lemma} 

\begin{proof}
From \eqref{eq-2-9} and \eqref{eq-2-11}, we can first obtain
\begin{equation}\label{eq-2-13}
	z_{(\Lambda)}^{n+1}-z^{n+1}=\Delta t\sum_{i=1}^sb_i\big(k_{i,(\Lambda)}^n-k_i^n)=\Delta t\sum_{i=1}^sb_i(g_{1,i}^n+g_{2,i}^n),
\end{equation}
where 
\begin{equation}\label{eq-2-14}
	g_{1,i}^n=f_1\big(z_{i,(\Lambda)}^n)-f_1(z_i^n\big),\quad g_{2,i}^n=f_2\big(z_{i,(\Lambda-1)}^n\big)-f_2\big(z_i^n\big).
\end{equation}
Note that $z_{i,(m)}^n$ and $z_i^n$ can be considered functions of the time step $\Delta t$. For $\Delta t = 0$, it follows directly that
\begin{equation*}
	z_{i,(m)}^n=z_i^n=z^n,\quad i=1,\cdots,s,~m=0,1,\cdots,\Lambda.
\end{equation*}
Differentiating the first two equations in \eqref{eq-2-9} and \eqref{eq-2-10} with respect to $\Delta t$ and applying the chain rule yields
\begin{equation*}
	\frac{d^{q}}{d(\Delta t)^{q}}\big(z_{i,(m)}^n-z_i^n\big)\bigg|_{\Delta t=0}=0,\quad q=0,1,\cdots,m.
\end{equation*}
Here the number $q$ denotes the order of derivatives. It immediately leads to 
\begin{equation}\label{eq-2-15}
	z_{i,(m)}^n-z_i^n=\mathcal{O}\big((\Delta t)^{m+1}\big),\quad m=\Lambda-1~\mbox{or}~\Lambda.
\end{equation}
Substituting \eqref{eq-2-15} into \eqref{eq-2-14} and considering the smoothness of $f_1$ and $f_2$, we obtain the following equation via the Taylor expansion as
\begin{equation*}
	g_{1,i}^n+g_{2,i}^n=\mathcal{O}\big((\Delta t)^{\Lambda+1}\big)+\mathcal{O}\big((\Delta t)^\Lambda\big)=\mathcal{O}\big((\Delta t)^\Lambda\big).
\end{equation*}
Combining \eqref{eq-2-13} with the identity $\sum_{i=1}^sb_i=1$, we have
\begin{equation*}
	z_{(\Lambda)}^{n+1}-z^{n+1}=\mathcal{O}\big((\Delta t)^{\Lambda+1}\big).
\end{equation*}
This completes the proof.
\end{proof}

\begin{theorem}\label{Th-2-4}
Given that the underlying RK method is of order $p$, the scheme \eqref{eq-2-10}-\eqref{eq-2-11} possesses an order of $\min\{p,\Lambda\}$, i.e.,
\begin{equation}\label{eq-2-16}
	z_{(\Lambda)}^{n+1}-z(t_{n+1})=\mathcal{O}\big((\Delta t)^{\min\{p,\Lambda\}+1}\big).
\end{equation}
\end{theorem}

\begin{proof}
If the underlying RK method is of order $p$, it holds for \eqref{eq-2-9} that
\begin{equation*}
	z^{n+1}-z(t_{n+1})=\mathcal{O}\big((\Delta t)^{p+1}\big).
\end{equation*} 
Combining this result with \eqref{eq-2-12} yields \eqref{eq-2-16}, which completes the proof.
\end{proof}

\subsubsection{The high-order Lagrange multiplier scheme}
In general, the above prediction-correction scheme \eqref{eq-2-10}-\eqref{eq-2-11} typically fails to conserve the energy. To address this issue, we employ it to develop a high-order Lagrange multiplier scheme. We then analyze the original energy conservation and numerical accuracy of the proposed scheme. Specifically, a high-order prediction step is applied to \eqref{eq-2-8}, followed by a high-order correction step applied to the Lagrange multiplier system \eqref{eq-2-1}. This procedure yields the high-order linear scheme as follows:
\begin{algorithm}[H]
\begin{itemize}
	\item[{\textbf{$\star$}}]\textbf{step1 (prediction):} set $z_{i,(0)}^n=z^n$ and compute $z_{i,(m+1)}^n$ from $m=0$ to $\Lambda-1$ by solving the linear system
	\begin{equation}\label{eq-2-17}
		\begin{aligned}
			z_{i,(m+1)}^n&=z^n+\Delta t\sum_{j=1}^sa_{ij}k_{j,(m+1)}^n,\\ k_{i,(m+1)}^n&=f_1\big(z_{i,(m+1)}^n\big)+f_2\big(z_{i,(m)}^n\big).
		\end{aligned}
	\end{equation}
	Then, we take $\tilde{z}_i^n=z_{i,(\Lambda)}^n$ for all $i=1,\cdots,s$.
	\item[{\textbf{$\star$}}]\textbf{step 2 (correction):} update $z^{n+1}$ by
	\begin{equation}\label{eq-2-18}
		\begin{aligned}
			z_i^n&=z^n+\Delta t\sum_{j=1}^sa_{ij}k_j^n,\\
			k_i^n&=f_1\big(z_i^n\big)+\lambda^nf_2\big(\tilde{z}_i^n\big),\\
			z^{n+1}&=z^n+\Delta t\sum_{i=1}^sb_ik_i^n,\\
			F(z^{n+1})&=F(z^n)+\lambda^n\Delta t\sum_{i=1}^sb_i\big(N'(\tilde{z}_i^n),k_i^n\big).
		\end{aligned}
	\end{equation}
\end{itemize}
\end{algorithm}
\noindent Note that $F(z^n)=\big(N(z^n),1\big)$. This scheme admits an energy conservation law of the following form.
\begin{theorem}\label{Th-2-5}
The scheme \eqref{eq-2-17}-\eqref{eq-2-18} satisfies the semi-discrete original energy conservation law
\begin{equation}\label{eq-2-19}
	\mathcal{H}^{n+1}=\mathcal{H}^n,\quad  \mathcal{H}^n=\frac{1}{2}\big(z^n,\mathcal{L}z^n\big)+\big(N(z^n),1\big).
\end{equation} 
\end{theorem}
\begin{proof}
The third equation in \eqref{eq-2-18} gives
\begin{equation}\label{eq-2-20}
	\mathcal{L}z^{n+1}=\mathcal{L}z^n+\Delta t\sum_{i=1}^sb_i\mathcal{L}k_i^n.
\end{equation}
Taking the inner product of \eqref{eq-2-20} with $z^{n+1}$ and using the symmetry of $\mathcal{L}$, we obtain
\begin{equation}\label{eq-2-21}
	\begin{aligned}
		\big(z^{n+1},\mathcal{L}z^{n+1}\big)&=\big(z^n,\mathcal{L}z^n\big)+2\Delta t\sum_{i=1}^sb_i\big(k_i^n,\mathcal{L}z^n\big)\\
		&\quad~+(\Delta t)^2\sum_{i,j=1}^sb_ib_j\big(k_i^n,\mathcal{L}k_j^n\big).
	\end{aligned}
\end{equation}
From the first equation of \eqref{eq-2-18}, it follows that
\begin{equation*}
	\mathcal{L}z^n=\mathcal{L}z_i^n-\Delta t\sum_{j=1}^sa_{ij}\mathcal{L}k_j^n.	
\end{equation*}	
Substituting this into the second term of \eqref{eq-2-21} gives
\begin{equation}\label{eq-2-22}
	\begin{aligned}
		\big(z^{n+1},\mathcal{L}z^{n+1}\big)&=\big(z^n,\mathcal{L}z^n\big)+2\Delta t\sum_{i=1}^sb_i\big(k_i^n,\mathcal{L}z_i^n\big),\\
		&~\quad+(\Delta t)^2\sum_{i,j=1}^s\big(b_ib_j-b_ia_{ij}-b_ja_{ji}\big)\big(k_i^n,\mathcal{L}k_j^n\big).
	\end{aligned}
\end{equation}
The RK symplecticity condition \cite{hairer-06-GNI-ODE} 
\begin{equation}\label{eq-2-23}
	b_ib_j-b_ia_{ij}-b_ja_{ji}=0,\quad i,j=1,\cdots,s,
\end{equation}
implies the vanishing of the second row in \eqref{eq-2-22}. Taking the inner product of the second equation in \eqref{eq-2-18} with $\mathcal{L}z_i^n+\lambda^nN'(\tilde{z}_i^n)$, and noting that $\mathcal{S}$ is skew-adjoint, we derive
\begin{equation*}
\big(k_i^n,\mathcal{L}z_i^n\big)+\lambda^n\big(k_i^n,N'(\tilde{z}_i^n)\big)=0.
\end{equation*}
Inserting it into \eqref{eq-2-22} yields
\begin{equation*}
	\frac{1}{2}\big(z^{n+1},\mathcal{L}z^{n+1}\big)=\frac{1}{2}\big(z^n,\mathcal{L}z^n\big)-\lambda^n\Delta t\sum_{i=1}^sb_i\big(k_i^n,N'(\tilde{z}_i^n)\big).
\end{equation*} 
Together with the fourth equation in \eqref{eq-2-18}, this result leads to \eqref{eq-2-19}, thus completing the proof.
\end{proof}

In the following, we discuss how to implement \eqref{eq-2-17}-\eqref{eq-2-18} efficiently. For each $m=0,1,\cdots,\Lambda-1$, we first rewrite \eqref{eq-2-17} as
\begin{equation}\label{eq-2-24}
	\textbf{z}_{(m+1)}^n=\mathcal{C}^{-1}\big(\textbf{z}^n+\Delta tA\textbf{f}_{2,(m)}^n\big),
\end{equation}
where $\mathcal{C}=I_s-\Delta t\mathcal{S}\mathcal{L}A$, $I_s$ denotes the $s\times s$ identity matrix, $A=(a_{ij})_{i,j=1}^s$ is the RK matrix, and the $s$-dimensional column vectors
\begin{equation*}
	\begin{aligned}
		\textbf{z}^n&=\big(z^n,z^n,\cdots,z^n\big)^T,\\
		\textbf{z}_{(m+1)}^n&=\big(z_{1,(m+1)}^n,\cdots,z_{s,(m+1)}^n\big)^T,\\
		\textbf{f}_{2,(m)}^n&=\big(f_2(z_{1,(m)}^n),\cdots,f_2(z_{s,(m)}^n)\big)^T.
	\end{aligned}
\end{equation*}
After completing the prediction step \eqref{eq-2-24}, we get
\begin{equation*}
\tilde{\textbf{z}}^n=\textbf{z}_{(\Lambda)}^n,
\end{equation*}
where $\tilde{\textbf{z}}^n=\big(\tilde{z}_1^n,\cdots,\tilde{z}_s^n\big)^T$. Inserting it into \eqref{eq-2-18}, we derive
\begin{equation*}
	\textbf{k}^n=L_1^n+\lambda^nR_1^n,
\end{equation*}
where 
\begin{equation}\label{eq-2-25}
	L_1^n=\mathcal{C}^{-1}f_1(\textbf{z}^n),\quad R_1^n=\mathcal{C}^{-1}f_2(\tilde{\textbf{z}}^n),\quad \textbf{k}^n=\big(k_1^n,\cdots,k_s^n\big)^T.
\end{equation}
Then, the final numerical solution is
\begin{equation}\label{eq-2-26}
	z^{n+1}=L_2^n+\lambda^nR_2^n,
\end{equation}
where 
\begin{equation}\label{eq-2-27}
	L_2^n=z^n+\Delta t\textbf{b}L_1^n,\quad R_2^n=\Delta t\textbf{b}R_1^n,\quad \textbf{b}=\big(b_1,b_2,\cdots,b_s\big).
\end{equation}
We can employ the Newton iteration with an initial value $1$ to solve the Lagrange multiplier $\lambda^n$ from the algebraic equation
\begin{equation}\label{eq-2-28}
	\big(N(L_2^n+\lambda^nR_2^n)-N(z^n),1\big)-\lambda^n\Delta t\textbf{b}\big(N'(\tilde{\textbf{z}}^n),L_1^n+\lambda^nR_1^n\big)=0.
\end{equation}
In summary, the scheme \eqref{eq-2-17}-\eqref{eq-2-18} can be implemented in the following steps:
\begin{algorithm}[H]
\begin{itemize}
	\item[{\textbf{$\star$}}]\textbf{step 1:} obtain $\tilde{\textbf{z}}^n$ from \eqref{eq-2-24};
	\item[{\textbf{$\star$}}]\textbf{step 2:} compute $L_1^n$, $R_1^n$ and $L_2^n$, $R_2^n$ from \eqref{eq-2-25} and \eqref{eq-2-27};
	\item[{\textbf{$\star$}}]\textbf{step 3:} search $\lambda^n$ by solving \eqref{eq-2-28};
	\item[{\textbf{$\star$}}]\textbf{step 4:} get $z^{n+1}$ by \eqref{eq-2-26} and then proceed to the next step.
\end{itemize}
\end{algorithm}

In the first two steps, this scheme requires solving only linear systems with constant coefficients. When combined with pseudo-spectral spatial discretization, these solution process can be accelerated using fast Fourier transform. Subsequent numerical tests show that the computational cost of the Lagrange multiplier is negligible, so the calculation of the proposed scheme is efficient.

\subsection{Convergence of the Lagrange multiplier scheme}
This subsection is devoted to a convergence analysis of the Lagrange multiplier scheme. Following the work of \cite{calvo-06-RK-SISC,calvo-10-Lyapunov-BIT}, we gives the following theorem.

\begin{theorem}\label{Th-2-6}
If $\big(N'(z^n),\frac{\delta\mathcal{H}}{\delta z}(z^n)\big)\neq0$, there exists a constant $\Delta t^*$ such that, for all $\Delta t\in(0,\Delta t^*]$, the algebraic equation \eqref{eq-2-28} locally defines a unique solution $\lambda^n=\lambda^n(\Delta t)$ in a neighborhood of $1$. Furthermore, if the underlying RK method is of order $p$, then the scheme \eqref{eq-2-17}-\eqref{eq-2-18} is of order \mbox{min}\{p,$\Lambda$\}, i.e.,
\begin{equation}\label{eq-2-29}
	z^{n+1}-z(t_{n+1})=\mathcal{O}\big((\Delta t)^{\min\{p,\Lambda\}+1}\big).
\end{equation}
\end{theorem}
\begin{proof}
Note that $\tilde{\textbf{z}}^n$, $L_1^n$, $R_1^n$, $L_2^n$ and $R_2^n$ in \eqref{eq-2-28} can be regarded as functions of $\Delta t$. It follows that
\begin{equation}\label{eq-2-30}
	\begin{aligned}
	\lim_{\Delta t\rightarrow 0}\tilde{\textbf{z}}^n&=\textbf{z}^n,\quad\lim_{\Delta t\rightarrow 0}L_1^n=f_1(\textbf{z}^n),\quad \lim_{\Delta t\rightarrow 0}R_1^n=f_2(\textbf{z}^n),\\
    \lim_{\Delta t\rightarrow 0}L_2^n&=z^n,\quad \lim_{\Delta t\rightarrow 0}R_2^n=0.
	\end{aligned}
\end{equation}
From \eqref{eq-2-28}, we introduce
\begin{equation}\label{eq-2-31}
	\lambda^n=\frac{\big(N(L_2^n+\lambda^nR_2^n)-N(z^n),1\big)}{\Delta t\textbf{b}\big(N'(\tilde{\textbf{z}}^n),L_1^n+\lambda^nR_1^n\big)}:=\varphi(\lambda^n,\Delta t).
\end{equation}
Using the L'H\^{o}pital's rule, along with the antisymmetry of $\mathcal{S}$ and \eqref{eq-2-30}, we derive
\begin{equation}\label{eq-2-32}
	\lim_{\Delta t\rightarrow 0}\lambda^n=\lim_{\Delta t\rightarrow 0}\varphi(\lambda^n,\Delta t)=\frac{\big(N'(z^n),f_1(z^n)\big)}{\big(N'(z^n),f_1(z^n)\big)}=1,
\end{equation}	
where $\big(N'(z^n),f_1(z^n)\big)=\big(N'(z^n),\frac{\delta\mathcal{H}}{\delta z}(z^n)\big)\neq0$. Analogous calculations show that
\begin{equation*}
	\lim_{\Delta t\rightarrow 0}\frac{\partial\varphi(\lambda^n,\Delta t)}{\partial\lambda^n}=\frac{0}{\big(N'(z^n),f_1(z^n)\big)}=0.
\end{equation*}	
This equation confirms a contraction property. Therefore, by the fixed-point theorem, for a sufficiently small $\Delta t^*$, \eqref{eq-2-31} possesses a unique solution $\lambda^n=\lambda^n(\Delta t)$ near $1$ for all $\Delta t\in(0,\Delta t^*]$. Subsequently, we can compute
\begin{equation}\label{eq-2-33}
	\lambda^n-1=\frac{\big(N(L_2^n+\lambda^nR_2^n)-N(z^n),1\big)-\Delta t\textbf{b}\big(N'(\tilde{\textbf{z}}^n),L_1^n+\lambda^nR_1^n\big)}{\Delta t\textbf{b}\big(N'(\tilde{\textbf{z}}^n),L_1^n+\lambda^nR_1^n\big)}:=\frac{\mu(\Delta t)}{\nu(\Delta t)}.
\end{equation}
Combining \eqref{eq-2-30}, \eqref{eq-2-32}, and the identity $\sum_{i=1}^s b_i = 1$, we obtain
\begin{equation*}
	\lim_{\Delta t\rightarrow 0}\textbf{b}\big(N'(\tilde{\textbf{z}}^n),L_1^n+\lambda^nR_1^n\big)=\big(N'(z^n),f_1(z^n)\big)\neq0.
\end{equation*}	
This result implies that
\begin{equation*}
	\nu(\Delta t)=\mathcal{O}(\Delta t).
\end{equation*}
Then, we define $z_{(\Lambda)}^{n+1}$ as the solution of
\begin{equation*}
	\begin{aligned}
		\hat{z}_i^n&=z^n+\Delta t\sum_{j=1}^sa_{ij}\hat{k}_j^n,\\
		\hat{k}_i^n&=f_1\big(\hat{z}_i^n\big)+f_2\big(\tilde{z}_i^n\big),\\
		z_{(\Lambda)}^{n+1}&=z^n+\Delta t\sum_{i=1}^sb_i\hat{k}_i^n,
	\end{aligned}
\end{equation*}
where $\tilde{z}_i^n=z_{i,(\Lambda)}^n$ is computed from \eqref{eq-2-17}. According to Theorem \ref{Th-2-4}, if the underlying RK method is of order $p$, it follows that
\begin{equation}\label{eq-2-34}
	\begin{aligned}
		z_{(\Lambda)}^{n+1}-z(t_{n+1})&=\mathcal{O}\big((\Delta t)^{\min\{p,\Lambda\}+1}\big),\\
		F\big(z(t_{n+1})\big)-\big(F(z^n)+\Delta t\textbf{b}\big(N'(\tilde{\textbf{z}}^n),\hat{\textbf{k}}^n\big)\big)&=\mathcal{O}\big((\Delta t)^{\min\{p,\Lambda\}+1}\big),
    \end{aligned}
\end{equation}	
where $\hat{\textbf{k}}^n=\big(\hat{k}_1^n,\cdots,\hat{k}_s^n\big)^T$, and we have $z_{(\Lambda)}^{n+1}=L_2^n+R_2^n$ and $\hat{\textbf{k}}^n=L_1^n+R_1^n$. Using the Taylor expansion and the first equation in \eqref{eq-2-34}, we get
\begin{equation*}
	F\big(z_{(\Lambda)}^{n+1}\big)-F\big(z(t_{n+1})\big)=\mathcal{O}\big((\Delta t)^{\min\{p,\Lambda\}+1}\big).
\end{equation*}
With \eqref{eq-2-32}, the above equations imply that as $\Delta t\rightarrow 0$,
\begin{equation*}
	\begin{aligned}
		\mu(\Delta t)&=\big(N(L_2^n+\lambda^nR_2^n)-N(z^n),1\big)-\Delta t\textbf{b}\big(N'(\tilde{\textbf{z}}^n),L_1^n+\lambda^nR_1^n\big),\\
		&=F\big(z_{(\Lambda)}^{n+1}\big)-F\big(z(t_{n+1})\big)+F\big(z(t_{n+1})\big)-\big(F(z^n)+\Delta t\textbf{b}\big(N'(\tilde{\textbf{z}}^n),\hat{\textbf{k}}^n\big)\big)\\
		&=\mathcal{O}\big((\Delta t)^{\min\{p,\Lambda\}+1}\big).
	\end{aligned}
\end{equation*}	
From \eqref{eq-2-33}, we then have
\begin{equation}\label{eq-2-35}
	\lambda^n-1=\frac{\mu(\Delta t)}{\nu(\Delta t)}=\mathcal{O}\big((\Delta t)^{\min\{p,\Lambda\}}\big).
\end{equation}
Thus, the accuracy of the scheme \eqref{eq-2-17}-\eqref{eq-2-18} can be obtained from
\begin{equation*}
	\begin{aligned}
		z^{n+1}-z(t_{n+1})&=L_2^n+\lambda^nR_2^n-z(t_{n+1})\\
		&=L_2^n+R_2^n-z(t_{n+1})+(\lambda^n-1)R_2^n,\\
		&=z_{(\Lambda)}^{n+1}-z(t_{n+1})+(\lambda^n-1)\Delta t\textbf{b}R_1^n.
	\end{aligned}
\end{equation*}	
In the limit as $\Delta t\rightarrow 0$, the equations \eqref{eq-2-30}, \eqref{eq-2-34} and \eqref{eq-2-35} collectively yield \eqref{eq-2-29}, which completes the proof.
\end{proof}

\begin{remark}
Theorem \ref{Th-2-6} requires that $\frac{\delta\mathcal{H}}{\delta z}$ not be orthogonal to $N'(z)$, this condition that is generically satisfied. Moreover, viewing the scheme \eqref{eq-2-17}-\eqref{eq-2-18} as a high-order generalization of the scheme LM-CN \eqref{eq-2-3}, this theorem also implies the second-order convergence of the latter.
\end{remark}

Finally, we employ the fourth- and sixth-order Gauss methods for numerical tests due to their excellent stability and high accuracy. The corresponding Butcher tableaux \cite{hairer-06-GNI-ODE} are displayed as follows:
\begin{equation*}
	\begin{array}{c|cc}
		\frac{1}{2}-\frac{\sqrt{3}}{6} & \frac{1}{4} & \frac{1}{4}-\frac{\sqrt{3}}{6}\\
		\frac{1}{2}+\frac{\sqrt{3}}{6} & \frac{1}{4}+\frac{\sqrt{3}}{6} & \frac{1}{4}\\
		\hline
		& \frac{1}{2} & \frac{1}{2}
	\end{array}~\mbox{and}~
	\begin{array}{c|cccc}
		\frac{1}{2}-\frac{\sqrt{15}}{10} & \frac{5}{36} & \frac{2}{9}-\frac{\sqrt{15}}{15} & \frac{5}{36}-\frac{\sqrt{15}}{30}\\
		\frac{1}{2} & \frac{5}{36}+\frac{\sqrt{15}}{24} & \frac{2}{9}  &\frac{5}{36}-\frac{\sqrt{15}}{24} \\
		\frac{1}{2}+\frac{\sqrt{15}}{10} & \frac{5}{36}+\frac{\sqrt{15}}{30} & \frac{2}{9}+\frac{\sqrt{15}}{15} & \frac{5}{36}\\
		\hline
		& \frac{5}{18} & \frac{4}{9} & \frac{5}{18}
	\end{array}.
\end{equation*}
We refer to the high-order Lagrange multiplier scheme \eqref{eq-2-17}-\eqref{eq-2-18} as LM-GAUSS. As the Gauss methods are symplectic, they satisfy the RK symplecticity condition \eqref{eq-2-23} in the proof of Theorem \ref{Th-2-5}. The diagonally implicit symplectic RK methods also serve as a promising alternative, capable of producing favorable numerical results. These methods will not be further explored in this paper. In summary, LM-GAUSS constitutes a novel class of linear energy-preserving schemes that can achieve arbitrary accuracy by increasing the RK stages.

\section{Experimental results}\label{sec:experiments}
This section presents numerical tests to evaluate the efficiency, accuracy, and structure-preserving properties of the proposed schemes. We compare all Lagrange multiplier schemes with the SAV scheme \eqref{eq-1-5} ($c_0=1$), the IEQ scheme \cite{cai-20-linear-MS-JCP,jiang-19-SG-IEQ-JSC} and HBVMs \cite{luigi-10-HBVM-JNAIAM}. For the prediction step \eqref{eq-2-17}, we set $\Lambda=p$, where $p$ is the order of the Gauss methods. The Lagrange multiplier is solved using the Newton iteration with a tolerance of $1e-12$. The spatial domain is discretized via the Fourier pseudo-spectral method \cite{shen-11-spectral}, which allows the application of fast Fourier transforms. All computations are carried out in Matlab R2016a with Intel Core i5-9500 CPU, 3.00 GHz and 8GB memory. 

\subsection{The Korteweg-de Vries equation}
We first test the KdV equation
\begin{equation*}
	u_t+\eta uu_x+\mu^2 u_{xxx}=0,\quad x\in\Omega\subset\mathbb{R}^1,~t\in (0,T],
\end{equation*}
where we set $\eta=1$ and periodic boundary conditions. The initial condition is computed from the exact solutions at $t=0$. The two different soliton solutions \cite{luigi-19-HBVM-KDV-JCAM} are shown as follows:
\begin{itemize}
\item[{(1)}]one-soliton wave solution: 
\begin{equation*}
	u(x,t)=3\gamma\Big(\mbox{sech}\big(\frac{\sqrt{\gamma}}{2\mu}(x-\gamma t)_{\Omega_1}\big)\Big)^2,~\mu^2=0.0013020833,~\gamma=\frac{1}{3},
\end{equation*}
where the notation
\begin{equation*}
	(\theta)_{\Omega_1}:=
	\left\{\aligned
	&x_R-\mbox{Rem}\left(x_R-\theta,x_R-x_L\right),\quad \mbox{if}\ \theta<x_L,\\
	&\theta,\quad\quad\quad\quad\quad\quad\quad\quad\quad\quad\quad\quad\quad\mbox{if}\ \theta\in\Omega_1,\\
	&x_L+\mbox{Rem}\left(\theta-x_L,x_R-x_L\right),\quad\mbox{if}\ \theta>x_R.
	\endaligned\right.
\end{equation*}
The mark Rem represents the remainder of the integer division of two parameters. We consider the spatial domain $\Omega_1=[x_L,x_R]=[-3,5]$, and the exact solution is temporally periodic with period $T=24$.
\item[{(2)}]two-soliton waves solution: 
\begin{equation*}
	u(x,t)=12\frac{k_1^2e^{\xi_1}+k_2^2e^{\xi_2}+2(k_2-k_1)^2e^{\xi_1+\xi_2}+\rho^2(k_2^2e^{\xi_1}+k_1^2e^{\xi_2})e^{\xi_1+\xi_2}}{(1+e^{\xi_1}+e^{\xi_2}+\rho^2e^{\xi_1+\xi_2})^2},
\end{equation*}
where we choose the parameter 
\begin{align*}
	\aligned
	&\mu=1,~k_1=0.4,~k_2=0.6,~\rho=(k_1-k_2)/(k_1+k_2),\\
	&\xi_1=k_1x-k_1^3t+4,~\xi_2=k_2x-k_2^3t+15,
	\endaligned
\end{align*}
and take $\Omega_2=[-40,40]$ as a tested domain.
\end{itemize}

\begin{table}[htbp]
\footnotesize
\caption{The KdV equation: temporal convergence tests for the one-soliton wave solution at time $T=1$ with $N=128$ and $\Delta t=0.002$.}\label{Tab-1}
\begin{center}
	\begin{tabular}{*{6}{l}}\hline
		Schemes & Time steps & $\Delta t$& $\Delta t/2$ & $\Delta t/2^2$ & $\Delta t/2^3$\\ \hline
		\multirow{2}*{IEQ-CN} & Error & 1.3135e-04 & 3.2978e-05 & 8.2360e-06 & 2.0585e-06\\
		& Order & - & 1.9939 & 2.0015 & 2.0004\\ \hline
		\multirow{2}*{SAV-CN} & Error & 1.3975e-04 & 3.5081e-05 & 8.7617e-06 & 2.1898e-06\\
		& Order & - & 1.9941 & 2.0014 & 2.0004\\ \hline
		\multirow{2}*{LM-CN} & Error & 1.4073e-04 & 3.5294e-05 & 8.8180e-06 & 2.2042e-06\\
		& Order & - & 1.9954 & 2.0009 & 2.0002\\ \hline
		\multirow{2}*{LM-GAUSS2} & Error & 9.3944e-08 & 5.8908e-09 & 3.6849e-10 & 2.3035e-11\\
		& Order & - & 3.9953 & 3.9988 & 3.9997\\ \hline
		\multirow{2}*{LM-GAUSS3} & Error & 5.2129e-11 & 8.1645e-13 & 1.3250e-14 &  2.1402e-16\\
		& Order & - & 5.9966 & 5.9453 & 5.9521\\ \hline
		\multirow{2}*{HBVM(3,3)} & Error & 5.2108e-11 & 8.1690e-13 & 1.3236e-14 & 2.1387e-16\\
		& Order & - & 5.9952 & 5.9476 & 5.9516\\ \hline
	\end{tabular}
\end{center}
\end{table}	
	
\begin{figure}[htbp]
\centering
\includegraphics[width=0.62\linewidth]{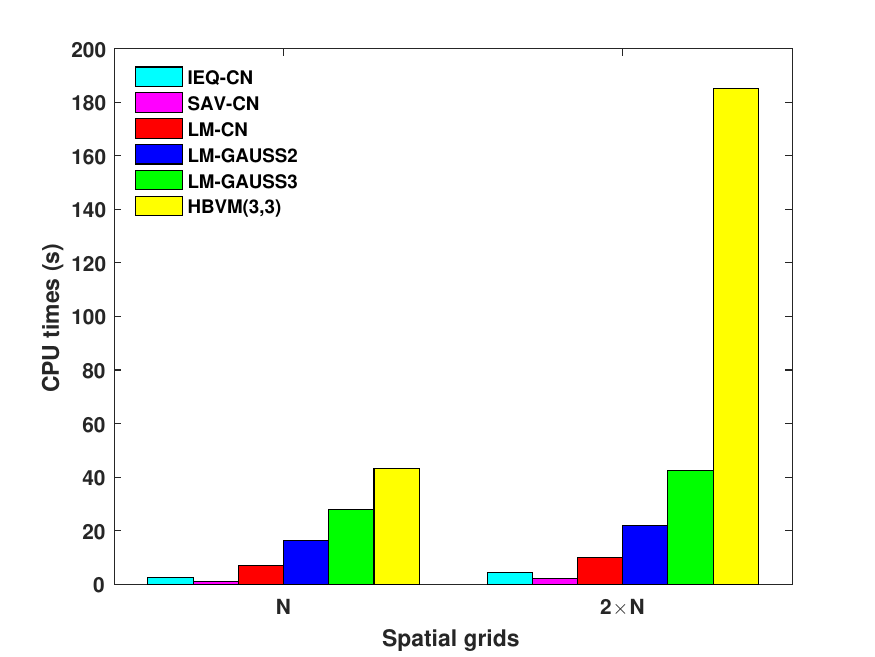}
\caption{The KdV equation: CPU times under different spatial grids for the one-soliton wave solution with $N=128$ and $\Delta t=0.002$ until $T=50$.}\label{Fig-1}
\end{figure}

\begin{figure}[htbp]
\centering
\includegraphics[width=0.62\linewidth]{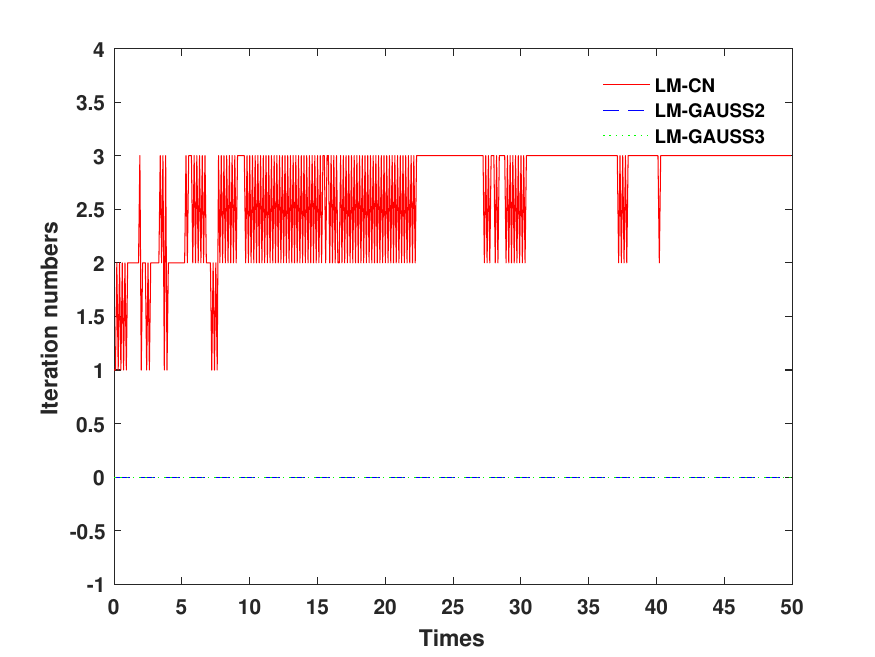}
\caption{The KdV equation: iteration numbers required to solve the Lagrange multiplier at each time step for the one-soliton wave solution with $N=128$ and $\Delta t=0.002$ until $T=50$.}\label{Fig-2}
\end{figure}

\begin{figure}[htbp]
\centering
\subfigure[numerical solutions of LM-CN]{
	\includegraphics[width=0.42\linewidth]{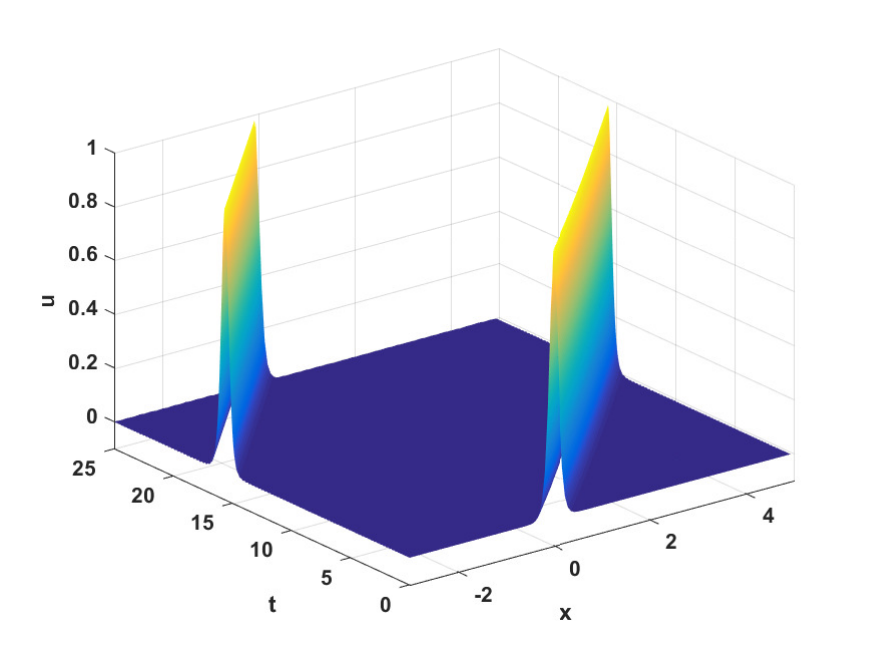}}\hspace{-2mm}
\subfigure[values of the Lagrange multiplier]{
	\includegraphics[width=0.42\linewidth]{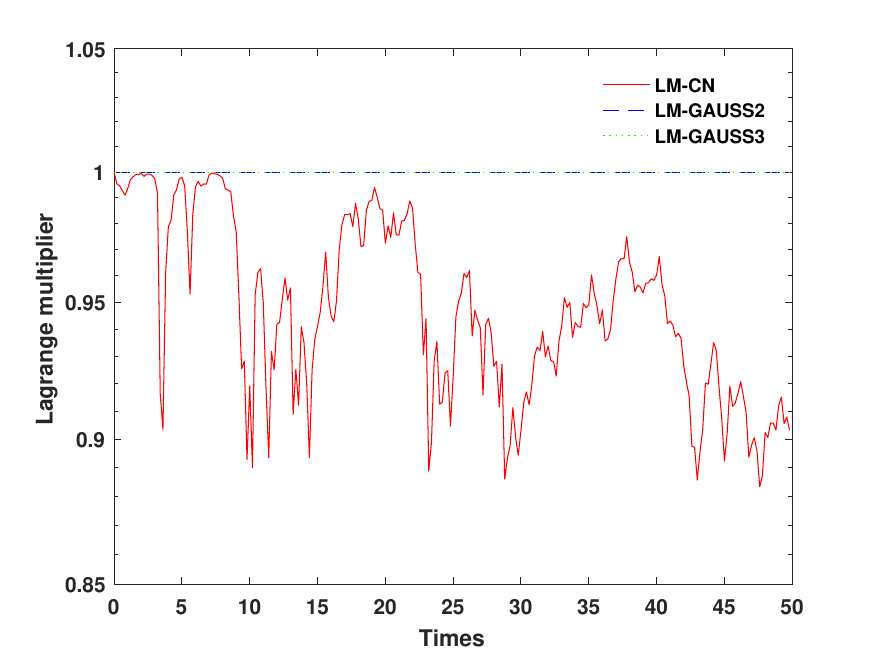}}\hspace{-2mm}
\subfigure[energy errors]{
	\includegraphics[width=0.42\linewidth]{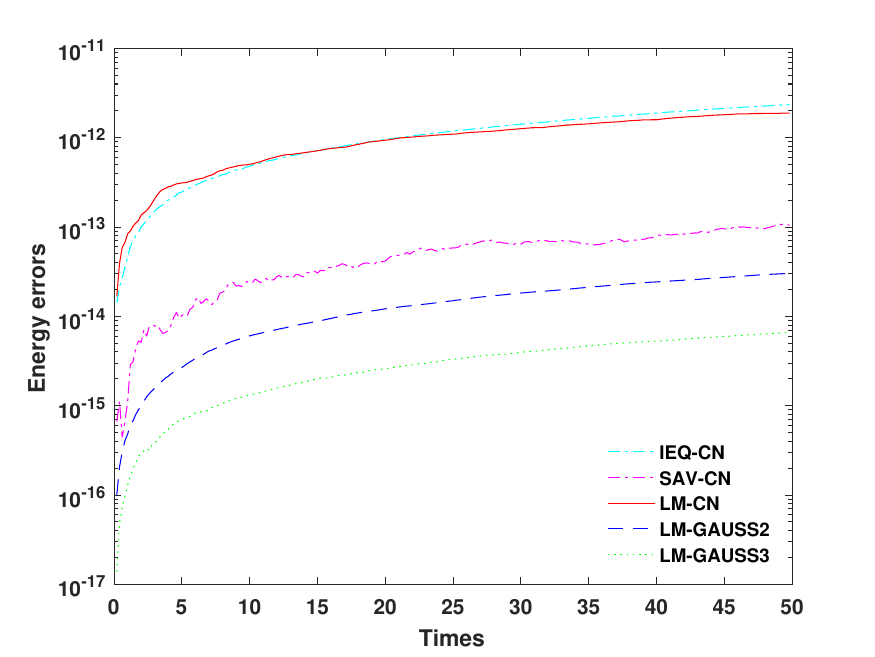}}\hspace{-2mm}
\subfigure[original energy errors]{
	\includegraphics[width=0.42\linewidth]{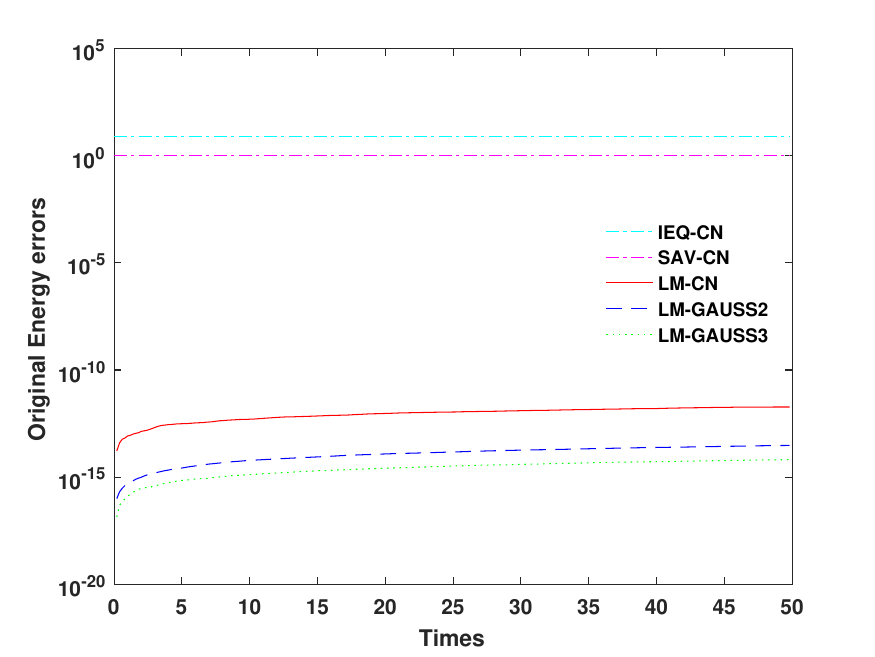}}\hspace{-2mm}
\caption{The KdV equation: numerical results for the one-soliton wave solution with $N=128$ and $\Delta t=0.002$ until $T=50$.}\label{Fig-3}
\end{figure}

\begin{figure}[htbp]
\centering
\subfigure[numerical solutions of LM-CN]{
	\includegraphics[width=0.42\linewidth]{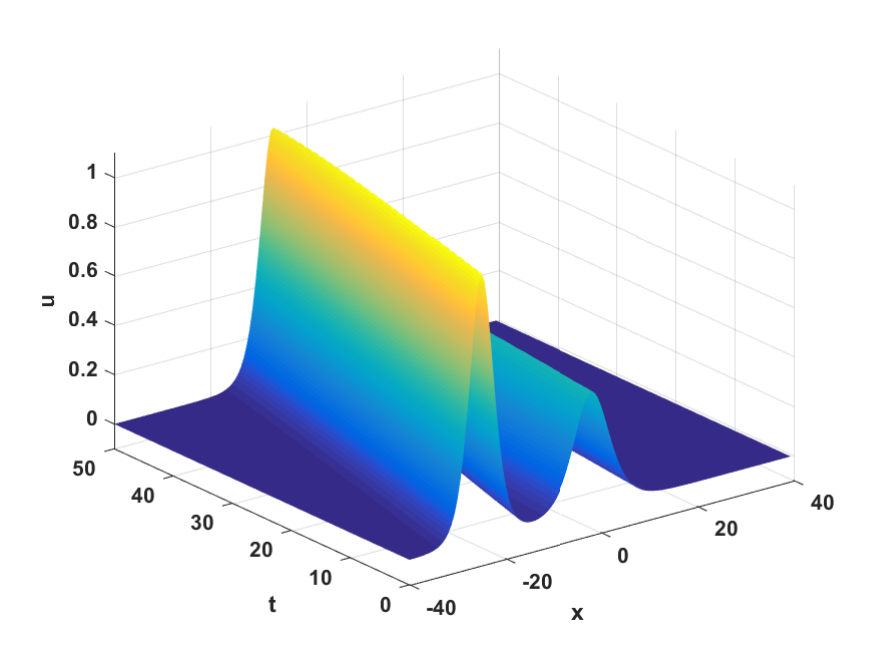}}\hspace{-2mm}
\subfigure[values of the Lagrange multiplier]{
	\includegraphics[width=0.42\linewidth]{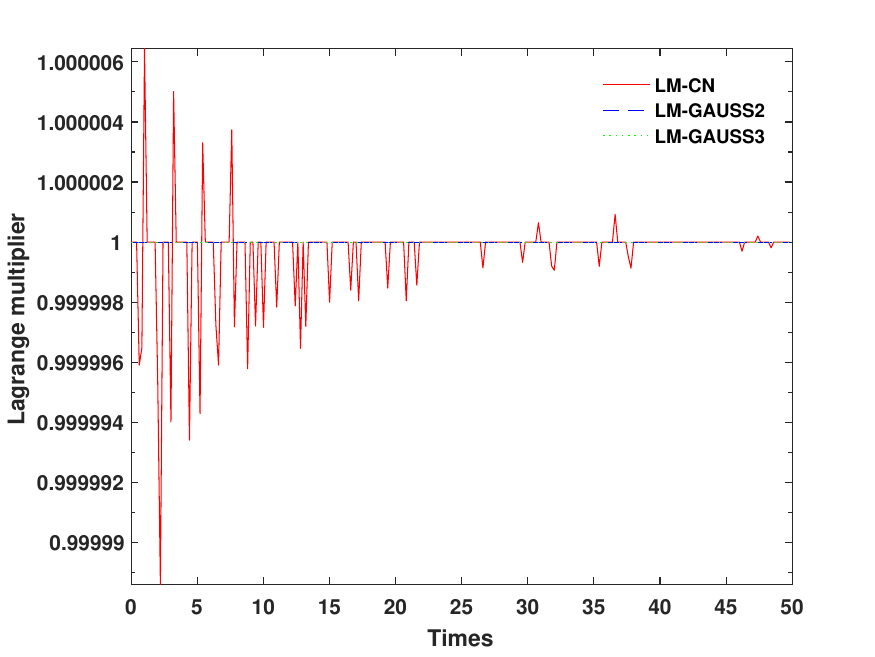}}\hspace{-2mm}
\subfigure[energy errors]{
	\includegraphics[width=0.42\linewidth]{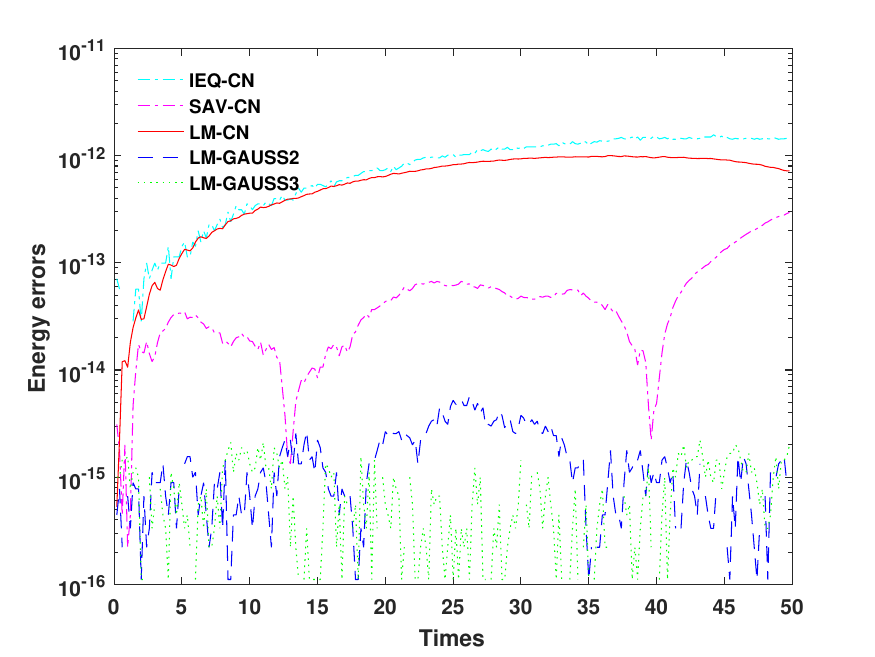}}\hspace{-2mm}
\subfigure[original energy errors]{
	\includegraphics[width=0.42\linewidth]{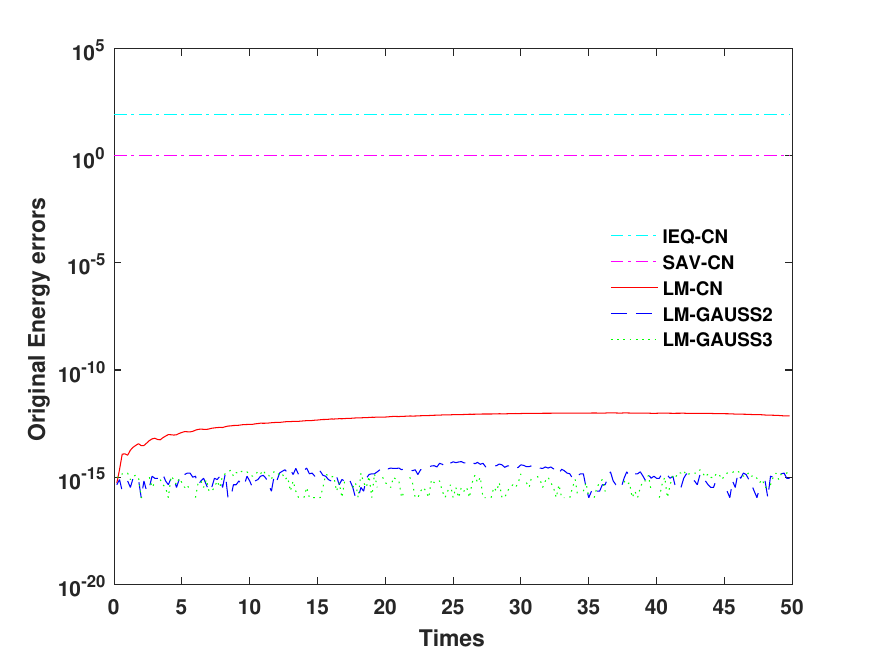}}\hspace{-2mm}
\caption{The KdV equation: numerical results for the two-soliton waves solution with $N=128$ and $\Delta t=0.002$ until $T=50$.}\label{Fig-4}
\end{figure}

Table \ref{Tab-1} confirms the temporal convergence rates of all schemes, which are in full agreement with the theoretical results in Theorem \ref{Th-2-6}. In particular, HBVM($3$,$3$) is the three-stage Gauss method. In Fig. \ref{Fig-1}, the comparison of CPU times reveals that the scheme LM-CN is remarkably efficient. Even though determining the Lagrange multiplier requires a cost of roughly two iterations, the runtime of LM-CN is comparable to that of the SAV and IEQ schemes. Owing to the high accuracy of the Gauss method, the number of iterations for the scheme LM-GAUSS is essentially zero in Fig. \ref{Fig-2}, demonstrating its high efficiency. This result also shows that the computational cost of the Lagrange multiplier is negligible. In Fig. \ref{Fig-3}, the one-soliton profile is well preserved and the numerical solution overlaps the exact waveform. All schemes accurately capture the two-soliton collision process depicted in Fig. \ref{Fig-4}. The Lagrange multiplier schemes preserve the original energy exactly over long-term simulations. In fact, the SAV and IEQ schemes conserve a modified energy that differs from the original energy by a constant. For instance, the SAV energy \eqref{eq-1-6} deviates by the constant $c_0=1$ from the original value.

\subsection{The nonlinear Schr\"odinger equation}
In this example, we consider the following NLS equation
\begin{equation*}
	\mbox{i}u_t+u_{xx}+\beta|u|^2u=0,\quad x\in\Omega\subset\mathbb{R}^1,\ t\in (0,T],
\end{equation*} 
where the mark $\mbox{i}$ is an imaginary unit and we set $\beta=1$. The computations are done on the space interval $\Omega=[-64,64]$. The two different soliton solutions \cite{gong-14-SP-MS-HPDEs} are tested as follows:
\begin{itemize}
	\item[{(1)}]one-soliton solution:
	\begin{equation*}
		u(x,0)=\frac{\sqrt{2}}{2}\exp\big(\mbox{i}\frac{x+20}{2}\big)\mbox{sech}\big(\frac{x+20}{2}\big),
	\end{equation*}
	\item[{(2)}]two-soliton solution:
	\begin{equation*}
		u(x,0)=\frac{\sqrt{2}}{2}\Big(\exp\big(\text{i}\frac{x+20}{2}\big)\text{sech}\big(\frac{x+20}{2}\big)+\exp\big(-\text{i}\frac{x-20}{2}\big)\text{sech}\big(\frac{x-20}{2}\big)\Big).
	\end{equation*}
\end{itemize}

\begin{figure}[htbp]
\centering
\subfigure[numerical solutions of LM-CN]{
	\includegraphics[width=0.42\linewidth]{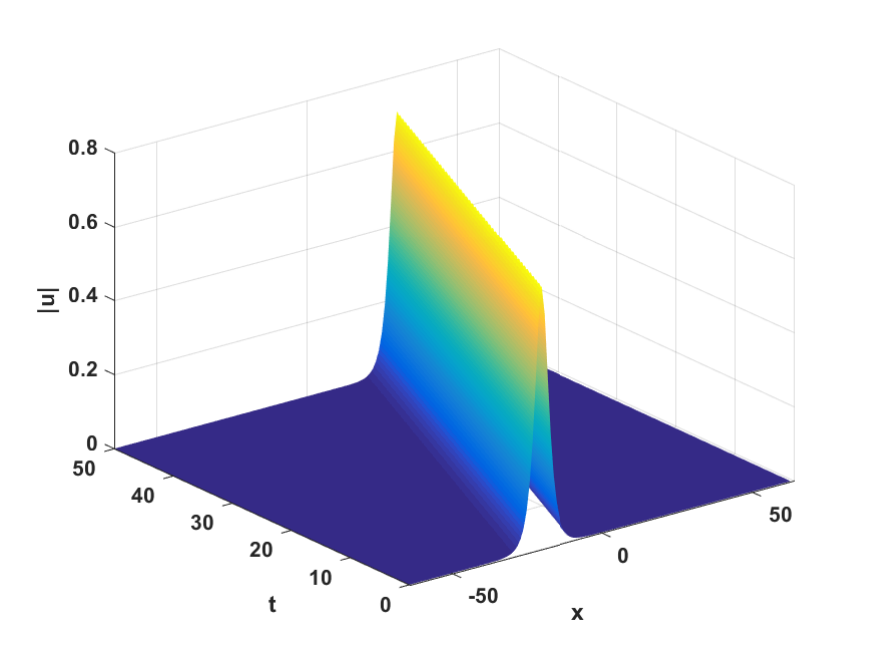}}\hspace{-2mm}
\subfigure[values of the Lagrange multiplier]{
	\includegraphics[width=0.42\linewidth]{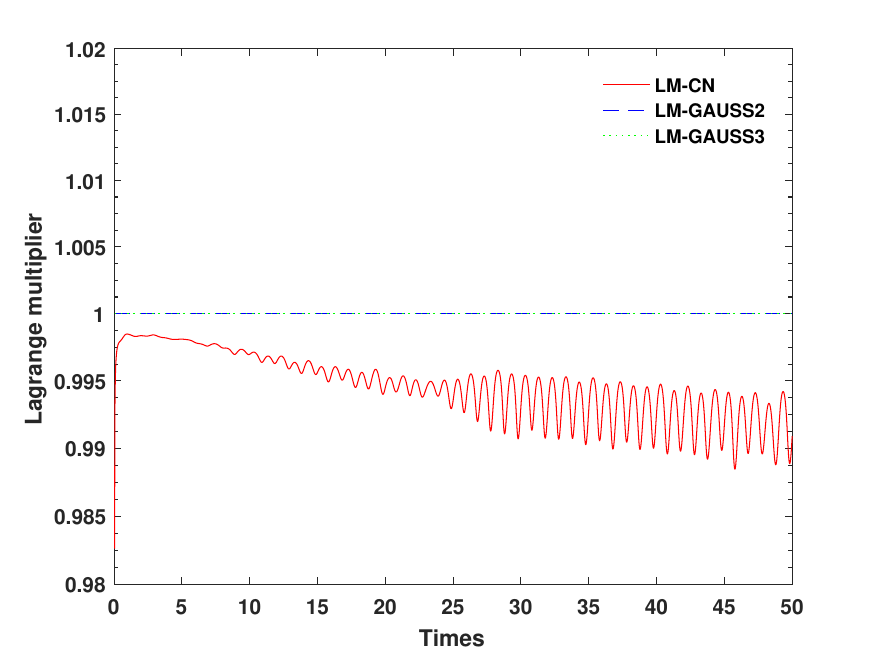}}\hspace{-2mm}
\subfigure[energy errors]{
	\includegraphics[width=0.42\linewidth]{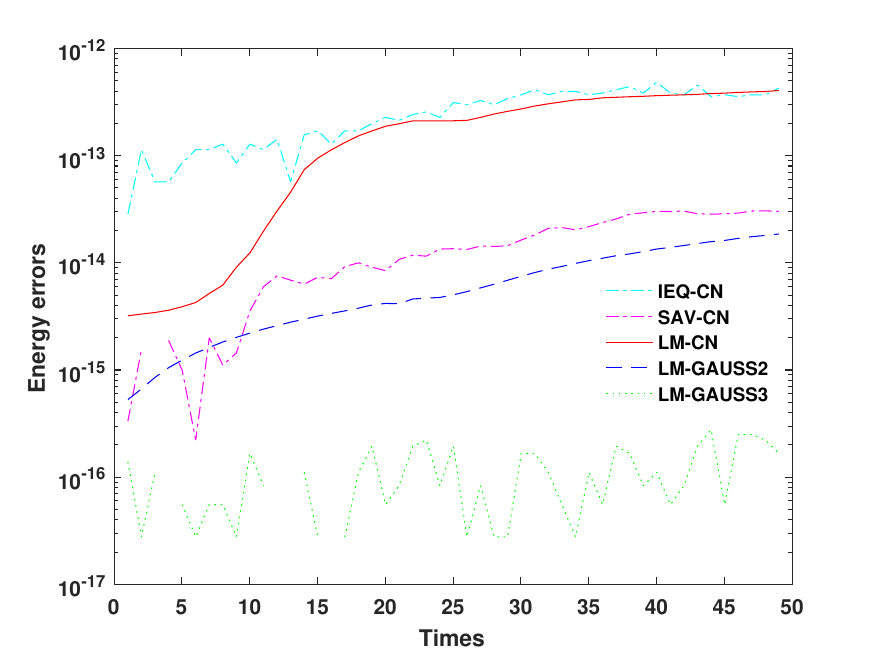}}\hspace{-2mm}
\subfigure[original energy errors]{
	\includegraphics[width=0.42\linewidth]{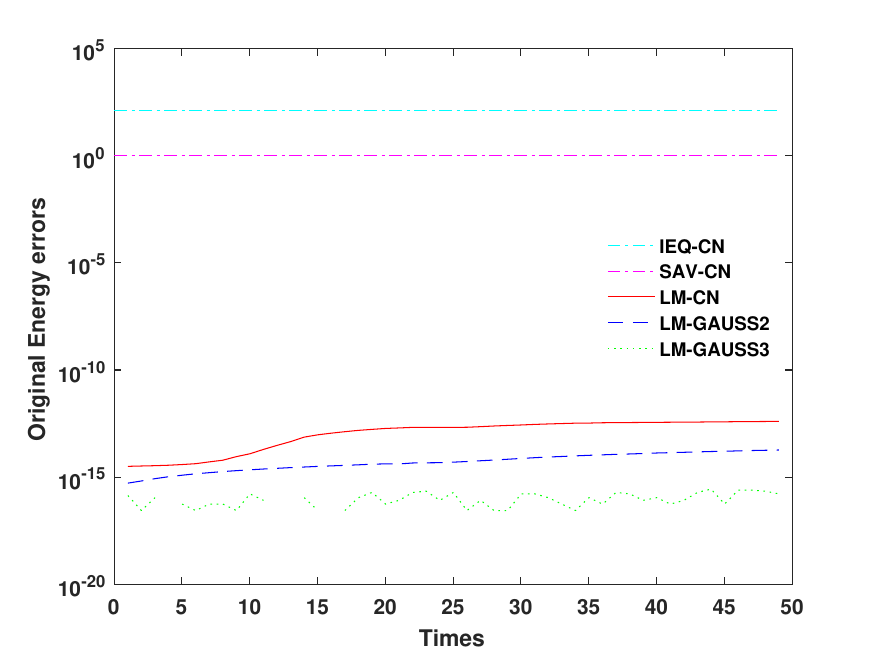}}\hspace{-2mm}
\caption{The NLS equation: numerical results for the one-soliton solution with $N=128$ and $\Delta t=0.02$ until $T=50$.}\label{Fig-5}
\end{figure}

\begin{figure}[htbp]
\centering
\subfigure[numerical solutions of LM-CN]{
	\includegraphics[width=0.42\linewidth]{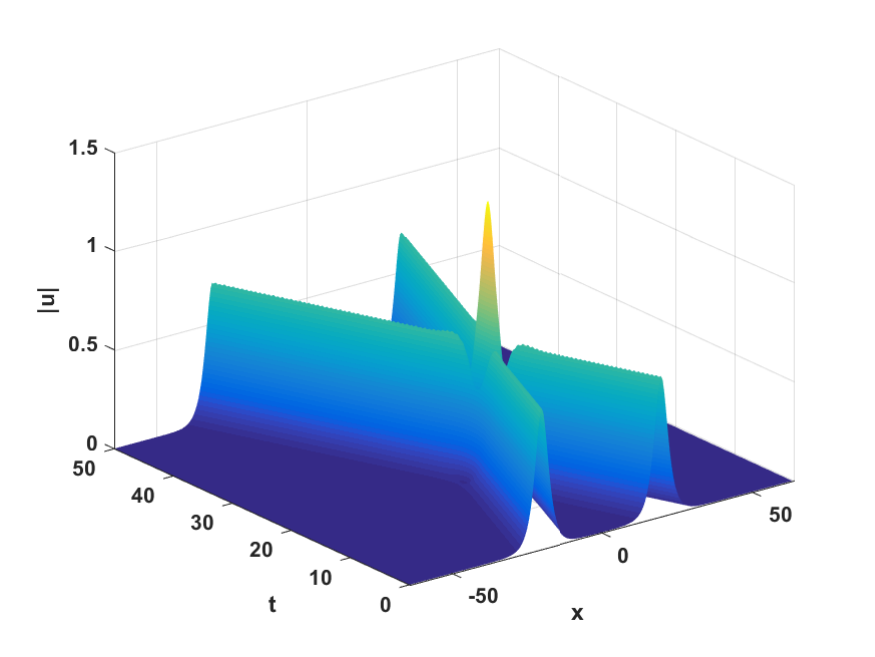}}\hspace{-2mm}
\subfigure[values of the Lagrange multiplier]{
	\includegraphics[width=0.42\linewidth]{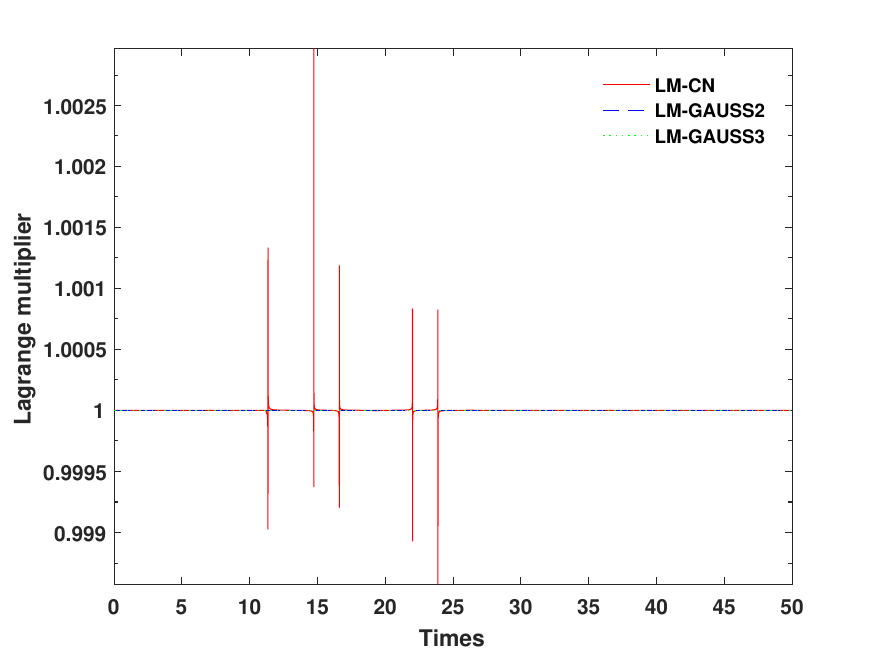}}\hspace{-2mm}
\subfigure[energy errors]{
	\includegraphics[width=0.42\linewidth]{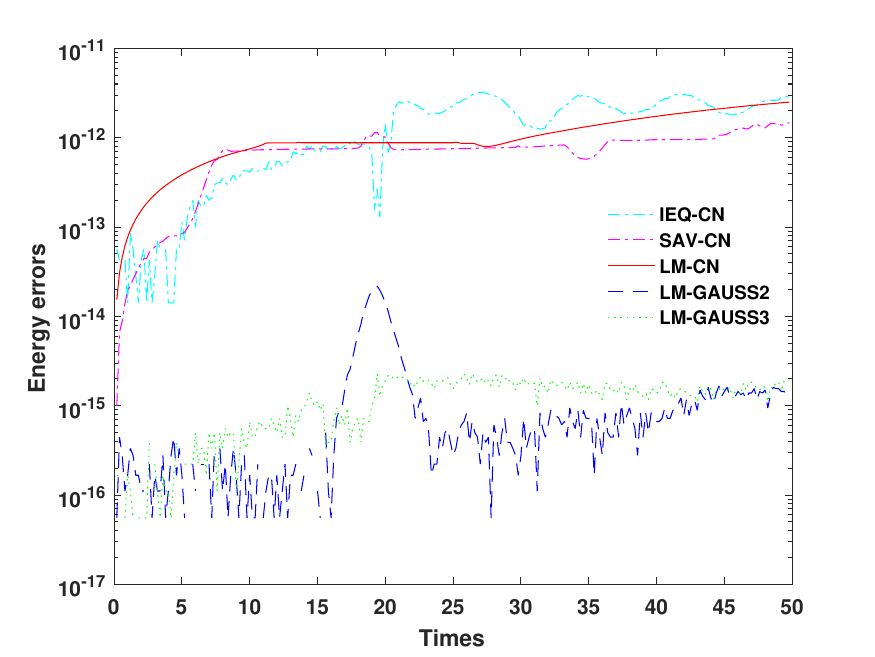}}\hspace{-2mm}
\subfigure[original energy errors]{
	\includegraphics[width=0.42\linewidth]{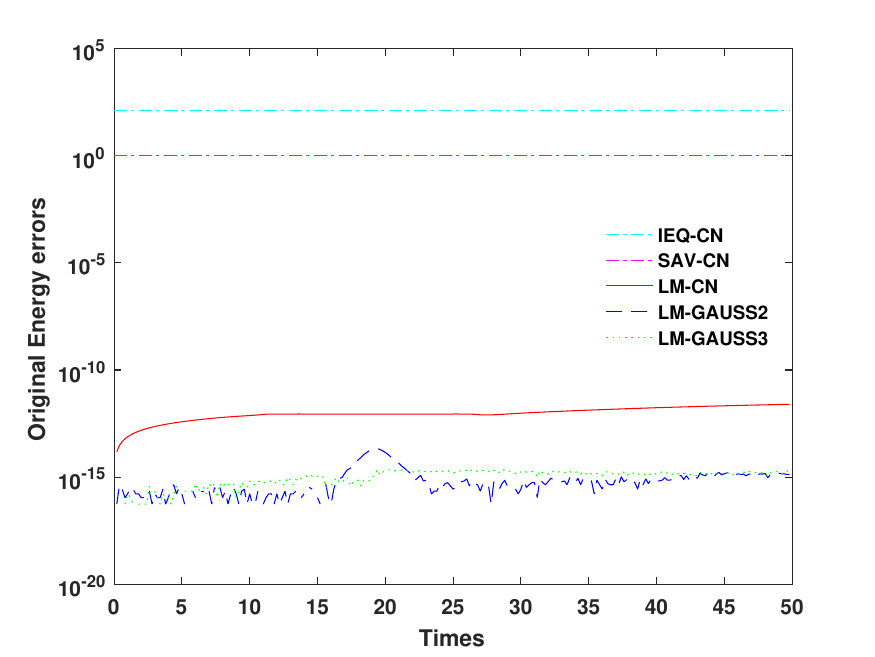}}\hspace{-2mm}
\caption{The NLS equation: numerical results for the two-soliton solution with $N=128$ and $\Delta t=0.02$ until $T=50$.}\label{Fig-6}
\end{figure}

As plotted in Fig. \ref{Fig-5}, the solitary wave propagates from left to right as expected over the time interval $[0,50]$ and accurately retains its shape. The interaction between two solitons during the same interval is shown in Fig. \ref{Fig-6}. The collision appears highly elastic, with both waves continuing to propagate in their original directions and same velocities. The Lagrange multiplier determined by the Newton iteration remains close to $1$, which is consistent with the analytical results in Theorem \ref{Th-2-6}. The proposed Lagrange multiplier schemes also exhibit remarkable long-time conservation of the original energy, in contrast to the auxiliary variable schemes, which fail to preserve this quantity.

\subsection{The sine-Gordon equation}
We consider the following two-dimensional SG equation
\begin{equation*}
	u_{tt}-u_{xx}-u_{yy}+\phi(x,y)\sin u=0,\quad (x,y)\in\Omega\subset\mathbb{R}^2,\ t\in (0,T],
\end{equation*} 
where we set $\phi(x,y)=1$ and periodic boundary conditions. The following two different initial conditions \cite{jiang-19-SG-IEQ-JSC} are carried out by selecting
\begin{itemize}
\item[{(1)}]circular ring soliton:
\begin{align*}
	\aligned
	&u(x,y,0)=4\tan^{-1}\big(\exp(3-\sqrt{x^2+y^2})\big),\\
	&u_t(x,y,0)=0,\quad \Omega_1=[-7,7]\times[-7,7],
	\endaligned
\end{align*}
\item[{(2)}]collision of two circular solitons:
\begin{align*}
	\aligned
	&u(x,y,0)=4\tan^{-1}\exp\big((4-\sqrt{(x+3)^2+(y+7)^2})/0.436\big),\\
	&u_t(x,y,0)=4.13~\mbox{sech}\big((4-\sqrt{(x+3)^2+(y+7)^2})/0.436\big),\\
	&\Omega_2=[-30,10]\times[-21,7].
	\endaligned
\end{align*}
\end{itemize} 

\begin{table}[htbp]
\footnotesize
\caption{The SG equation: temporal convergence tests for the circular ring soliton at time $T=1$ with $N=128$ and $\Delta t=0.02$.}\label{Tab-2}
\begin{center}
	\begin{tabular}{*{6}{l}}\hline
		Schemes & Time steps & $\Delta t$& $\Delta t/2$ & $\Delta t/2^2$ & $\Delta t/2^3$\\ \hline
		\multirow{2}*{IEQ-CN} & Error & 5.1630e-04 & 1.1983e-04 & 2.9663e-05 & 7.3733e-06\\
		& Order & - & 2.1072 & 2.0143 & 2.0083\\ \hline
		\multirow{2}*{SAV-CN} & Error & 5.1041e-04 & 1.2114e-04 & 2.9994e-05 & 7.4563e-06\\
		& Order & - & 2.0749 & 2.0140 & 2.0081\\ \hline
		\multirow{2}*{LM-CN} & Error & 5.4018e-04 & 1.2867e-04 & 3.1958e-05 & 7.9663e-06\\
		& Order & - & 2.0697 & 2.0095 & 2.0042\\ \hline
		\multirow{2}*{LM-GAUSS2} & Error & 3.8708e-06 & 2.4679e-07 & 1.5499e-08 & 9.6990e-10\\
		& Order & - & 3.9713 & 3.9930 & 3.9982\\ \hline
		\multirow{2}*{LM-GAUSS3} & Error & 1.3385e-08 & 2.1207e-10 & 3.3271e-12 & 5.5955e-14\\
		& Order & - & 5.9798 & 5.9941 & 5.8939\\ \hline
		\multirow{2}*{HBVM(3,3)} & Error & 1.3385e-08 & 2.1206e-10 & 3.3351e-12 & 5.1514e-14\\
		& Order & - & 5.9800 & 5.9906 & 6.0166\\ \hline
    \end{tabular}
\end{center}
\end{table}

\begin{figure}[htbp]
\centering
\includegraphics[width=0.62\linewidth]{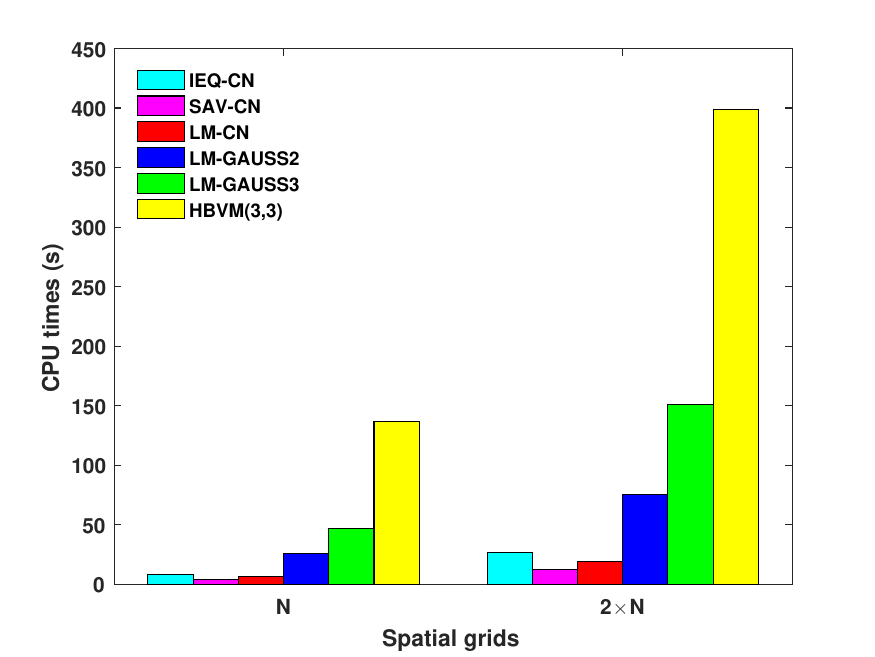}
\caption{The SG equation: CPU times under different spatial grids for the circular ring soliton with $N=128$ and $\Delta t=0.02$ until $T=50$.}\label{Fig-7}
\end{figure}

\begin{figure}[htbp]
\centering
\includegraphics[width=0.62\linewidth]{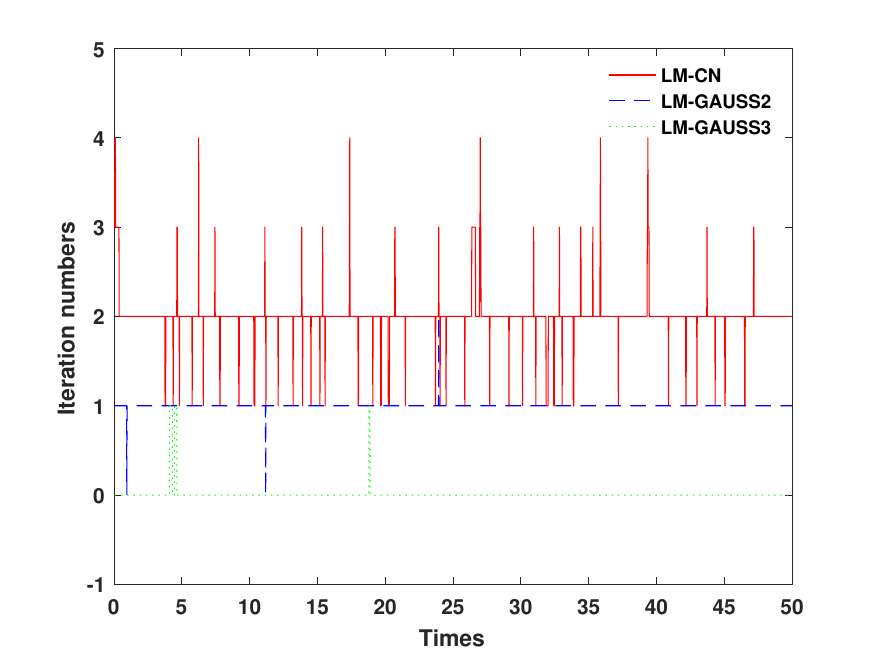}
\caption{The SG equation: iteration numbers required to solve the Lagrange multiplier at each time step for the circular ring soliton with $N=128$ and $\Delta t=0.02$ until $T=50$.}\label{Fig-8}
\end{figure}

\begin{figure}[htbp]
\centering
\includegraphics[width=0.32\linewidth]{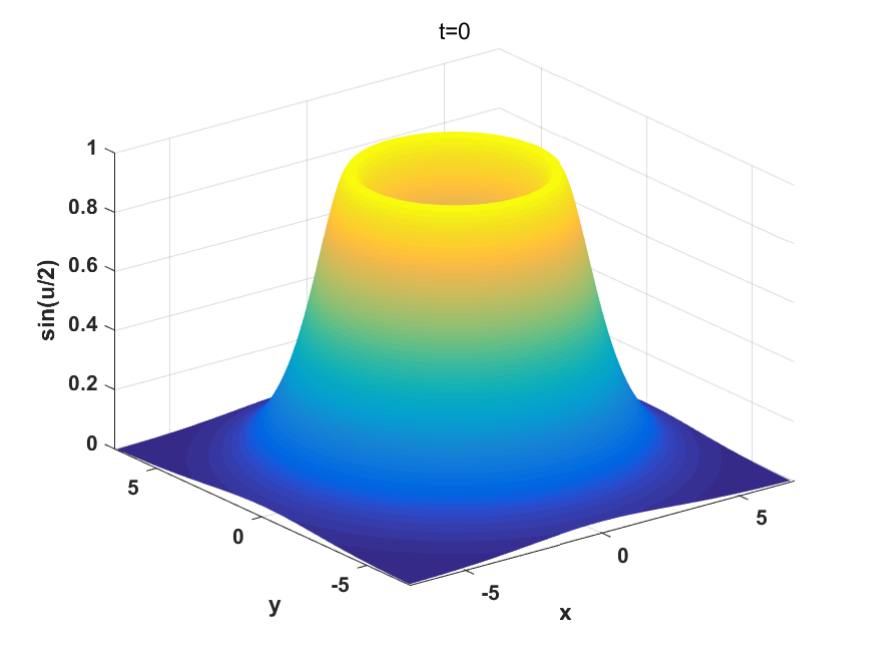}\hspace{-2mm}
\includegraphics[width=0.32\linewidth]{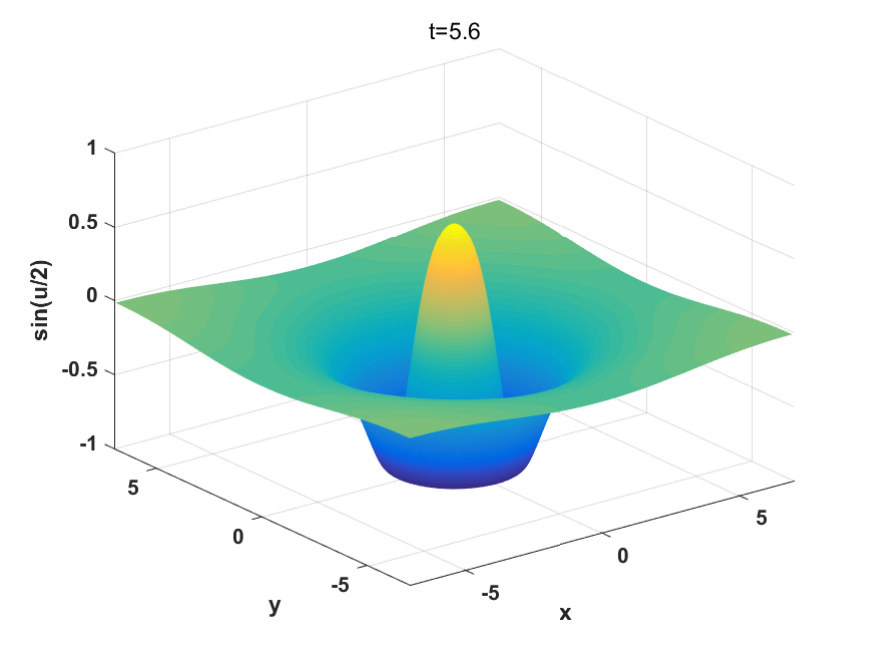}\hspace{-2mm}
\includegraphics[width=0.32\linewidth]{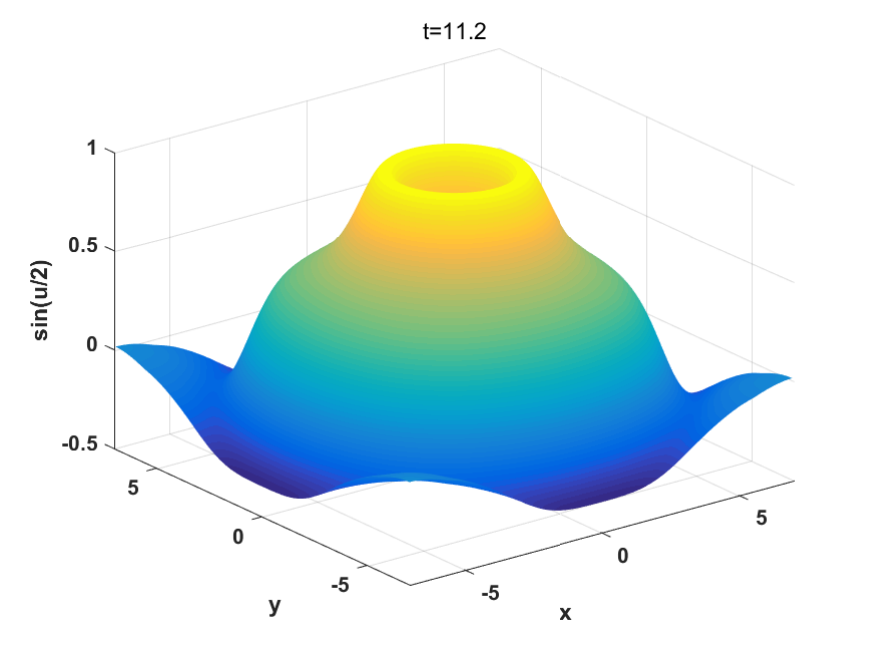}\hspace{-2mm}
\includegraphics[width=0.32\linewidth]{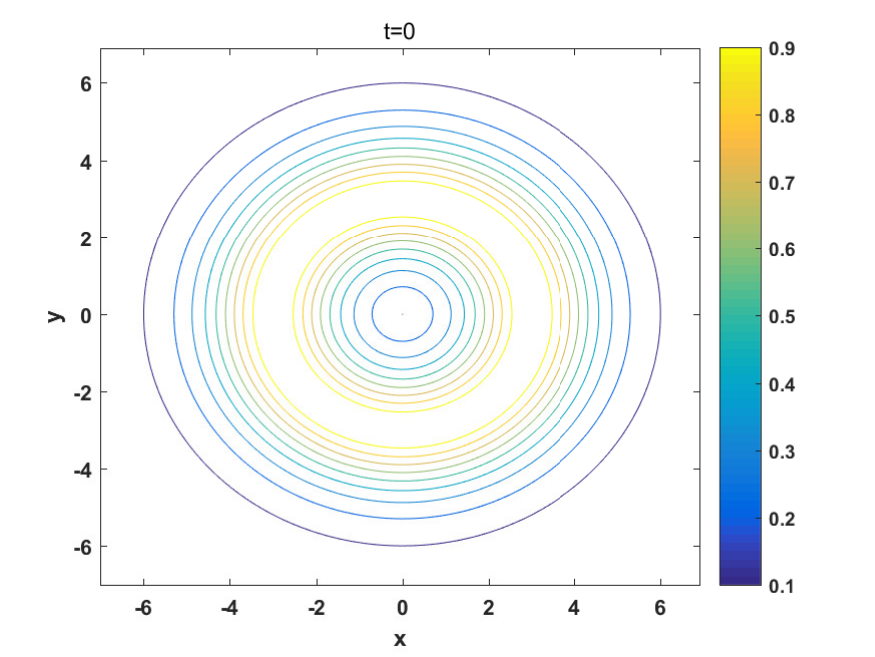}\hspace{-2mm}
\includegraphics[width=0.32\linewidth]{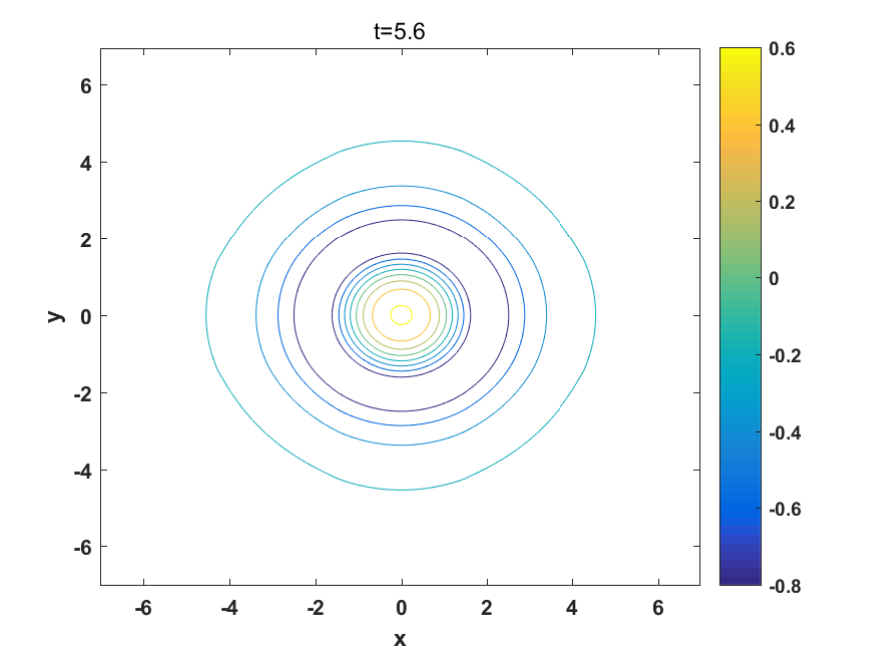}\hspace{-2mm}
\includegraphics[width=0.32\linewidth]{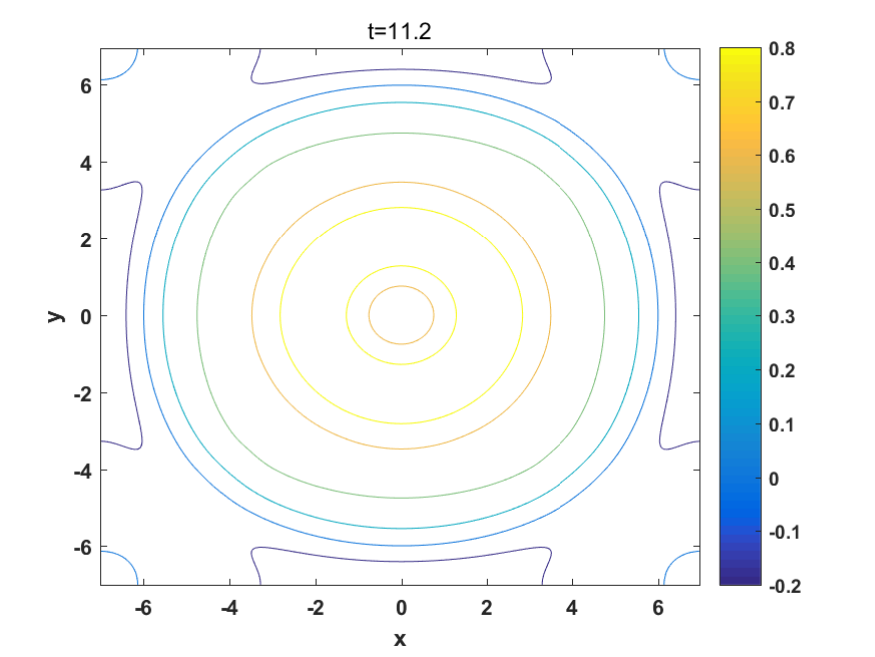}\hspace{-2mm}
\caption{Circular ring soliton: surfaces and contours of the initial condition and numerical solutions of LM-CN in terms of $sin(u/2)$ at different times with $N=128$ and $\Delta t=0.02$.}\label{Fig-9}
\end{figure}

\begin{figure}[htbp]
\centering
\includegraphics[width=0.32\linewidth]{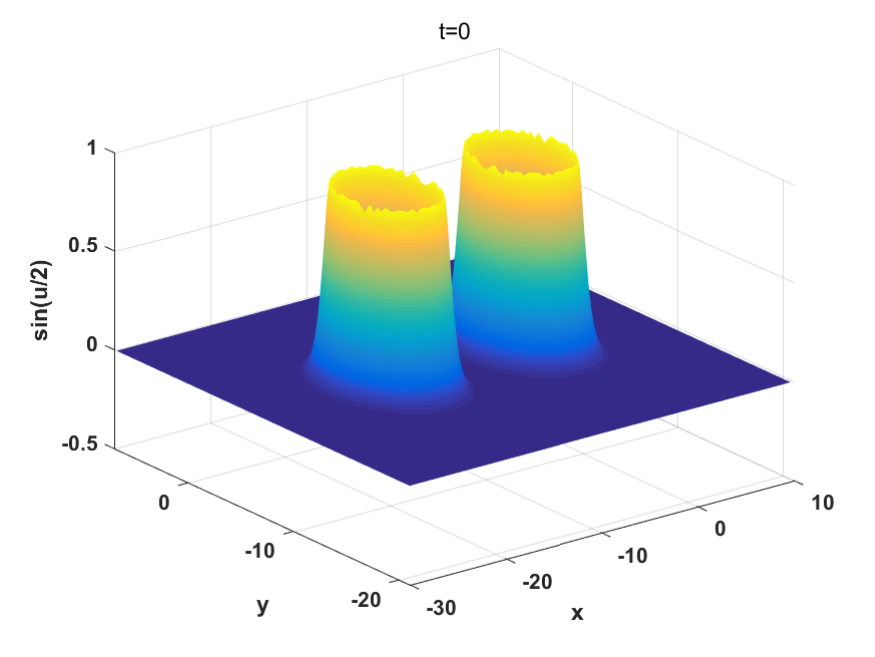}\hspace{-2mm}
\includegraphics[width=0.32\linewidth]{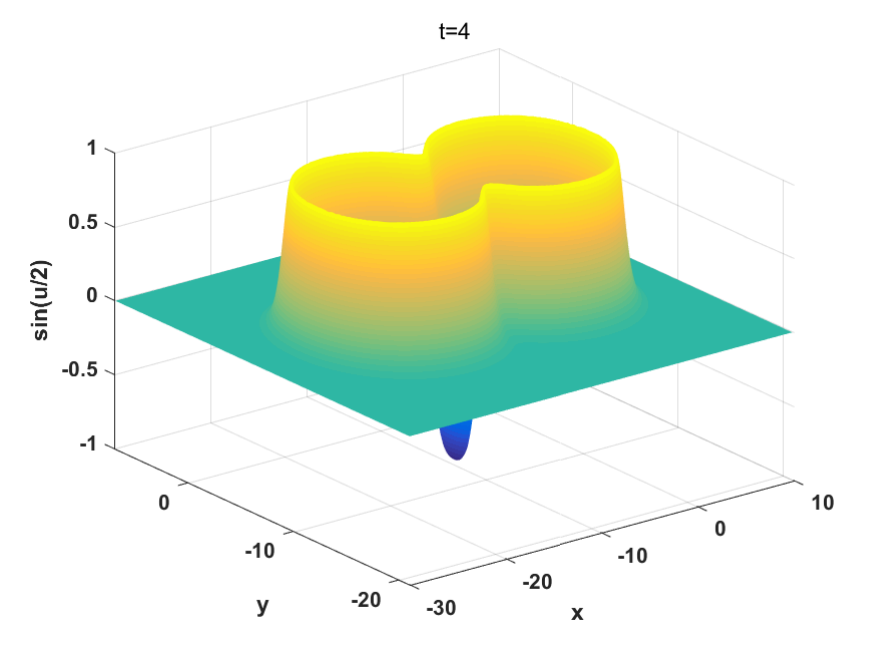}\hspace{-2mm}
\includegraphics[width=0.32\linewidth]{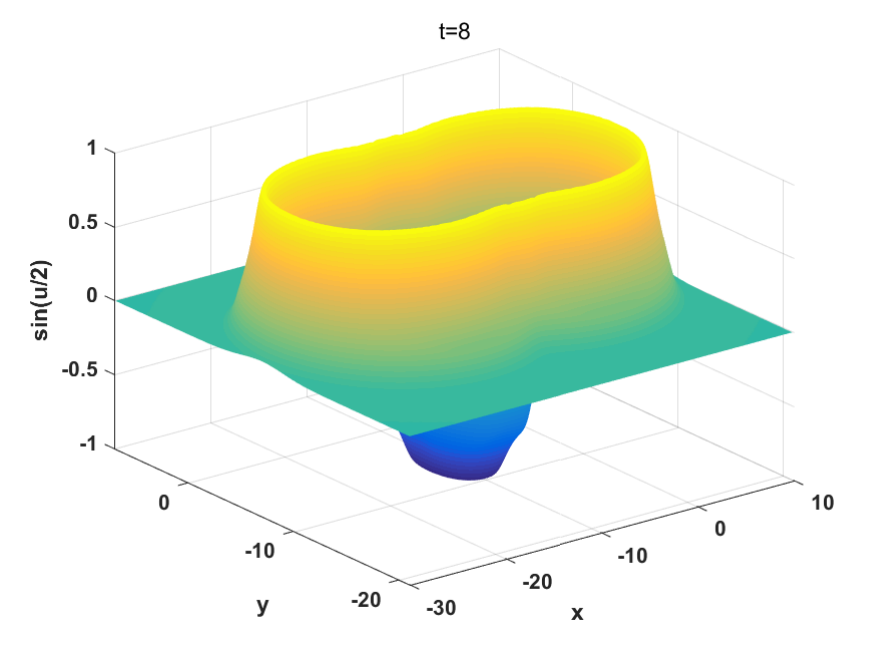}\hspace{-2mm}
\includegraphics[width=0.32\linewidth]{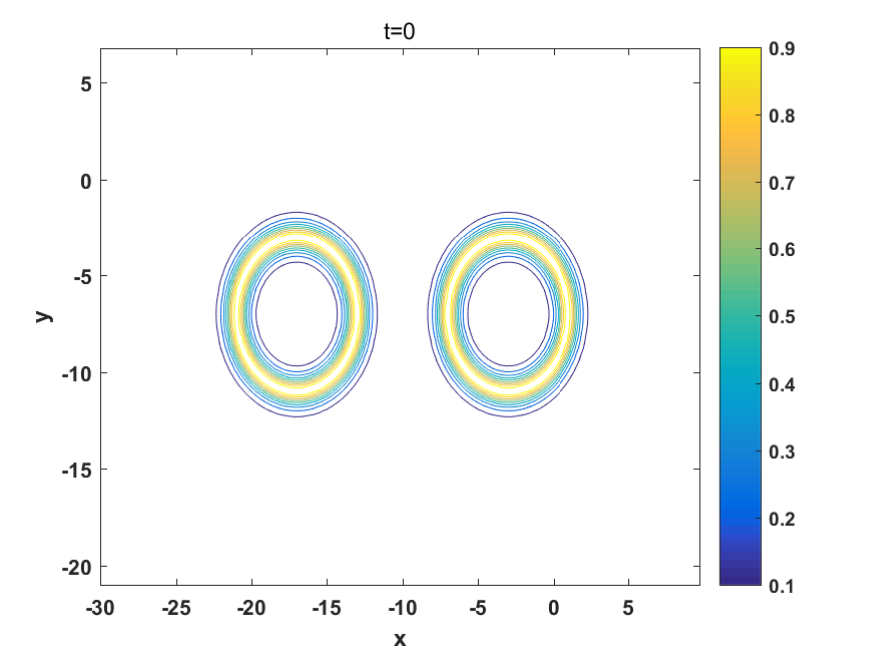}\hspace{-2mm}
\includegraphics[width=0.32\linewidth]{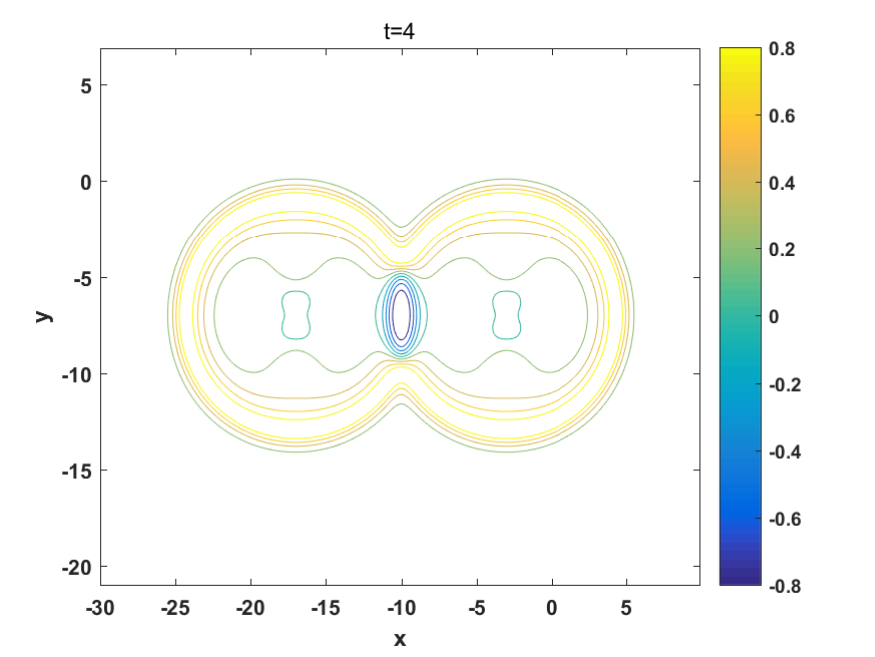}\hspace{-2mm}
\includegraphics[width=0.32\linewidth]{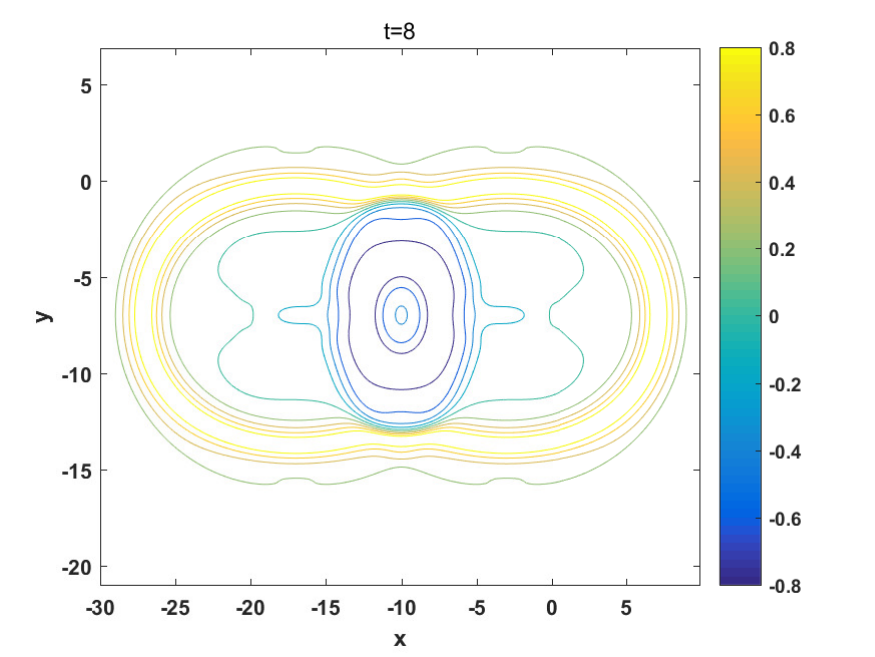}\hspace{-2mm}
\caption{Collision of two circular solitons: surfaces and contours of the initial condition and numerical solutions of LM-CN in terms of $sin(u/2)$ at different times with $N=128$ and $\Delta t=0.02$.}\label{Fig-10}
\end{figure}

\begin{figure}[htbp]
\centering
\includegraphics[width=0.32\linewidth]{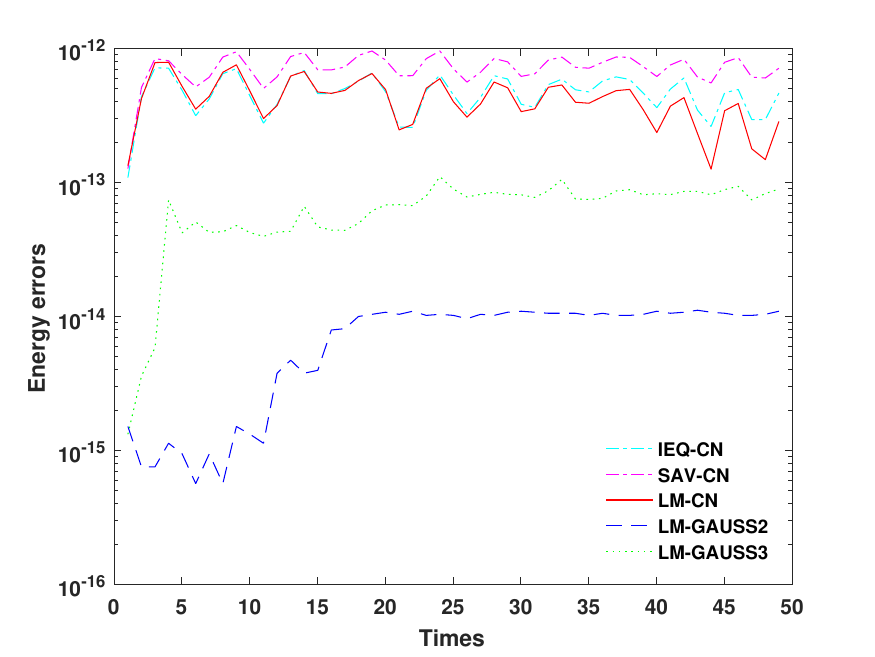}\hspace{-2mm}
\includegraphics[width=0.32\linewidth]{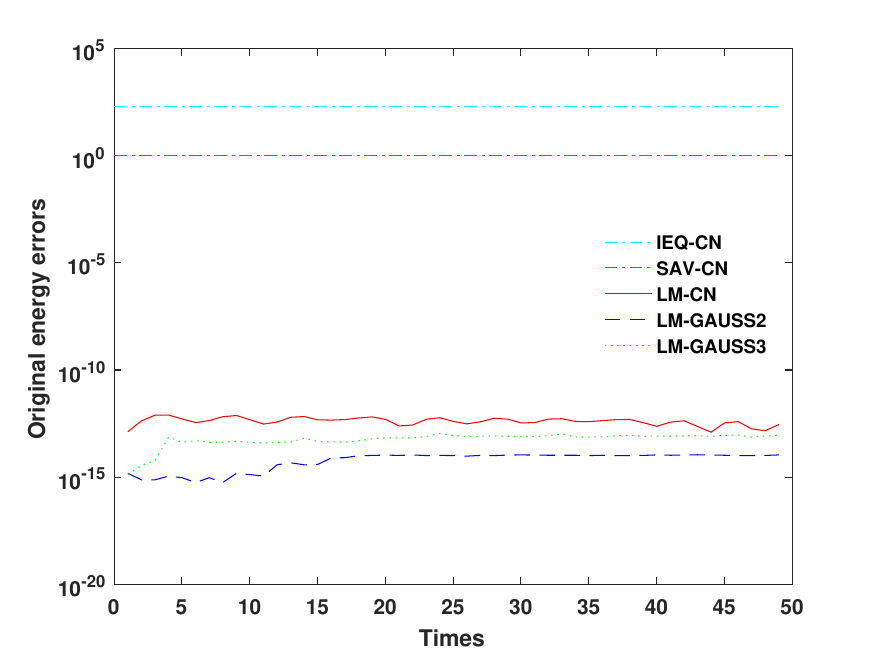}\hspace{-2mm}
\includegraphics[width=0.32\linewidth]{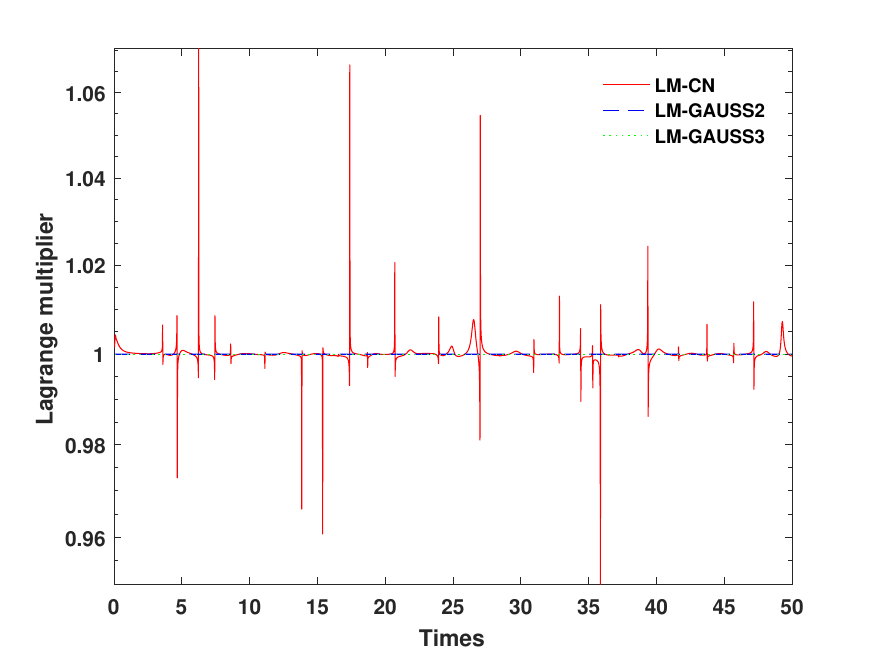}\hspace{-2mm}
\includegraphics[width=0.32\linewidth]{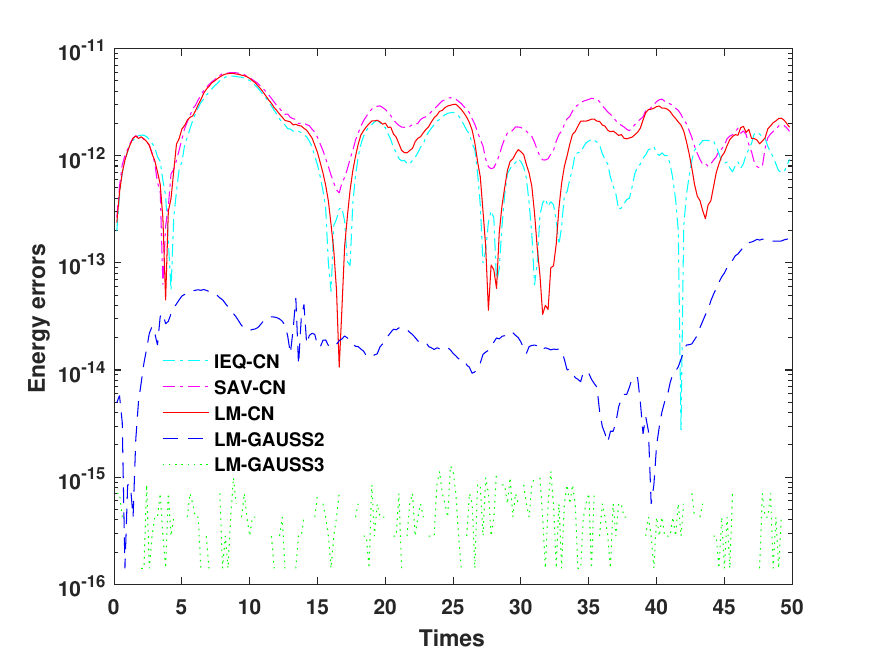}\hspace{-2mm}
\includegraphics[width=0.32\linewidth]{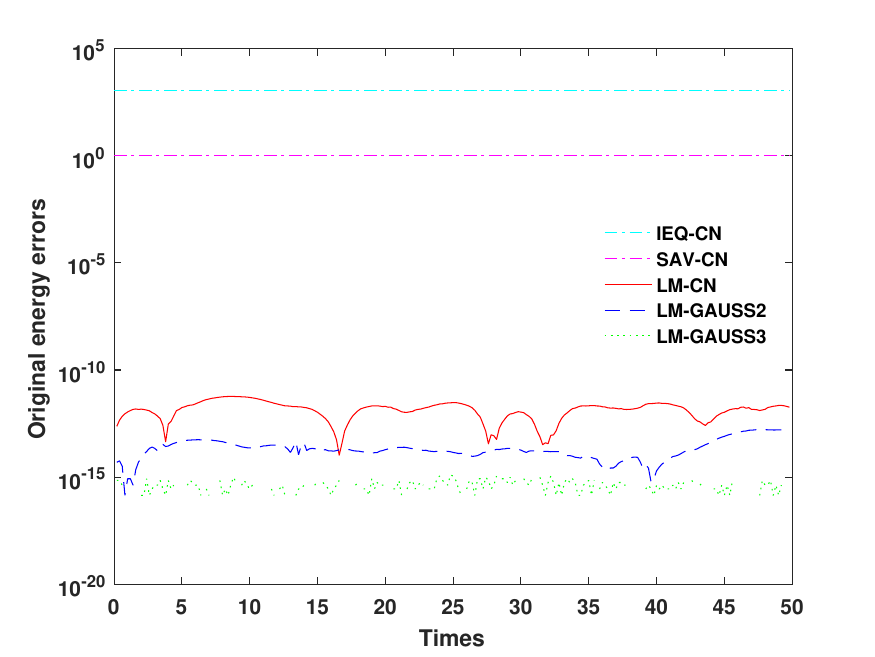}\hspace{-2mm}
\includegraphics[width=0.32\linewidth]{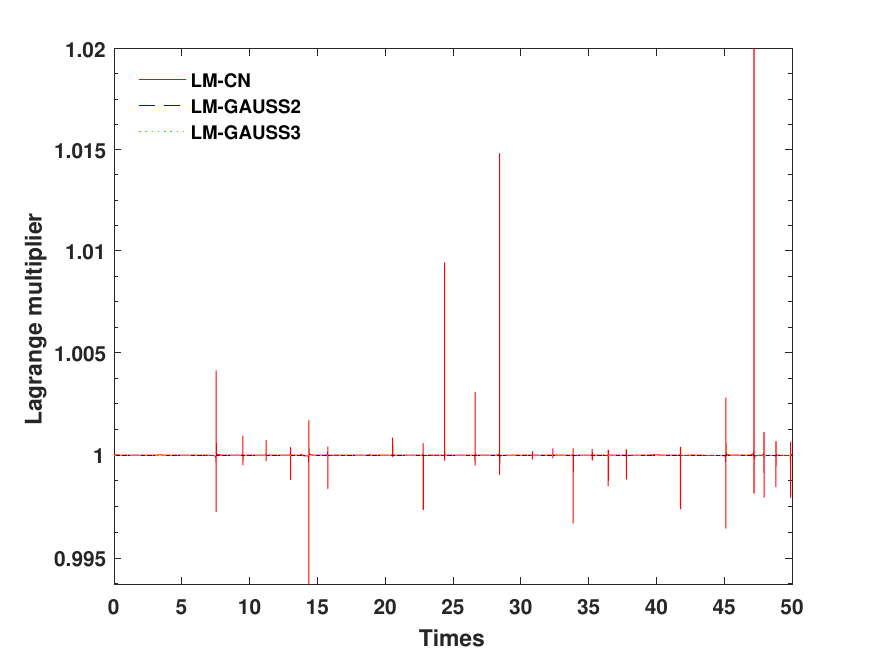}\hspace{-2mm}
\caption{The SG equation: Long-time energy errors and values of the Lagrange multiplier for the circular ring soliton (first row) and collision of two circular solitons (second row) with $N=128$ and $\Delta t=0.02$ until $T=50$.}\label{Fig-11}
\end{figure}

The superior performance of the proposed schemes extends to two-dimensional problems. Table \ref{Fig-2} shows that their convergence rates match the theoretical results. In Fig. \ref{Fig-7}, the runtime of LM-CN remains competitive with that of the auxiliary variable schemes even for high-dimensional cases, further demonstrating its high efficiency. All high-order schemes also perform efficiently in long-time computations. These schemes require a cost of roughly one iteration in Fig. \ref{Fig-8}, indicating a minimal computational cost. The values of the Lagrange multiplier are still near $1$, which again verifies the theoretical results. Fig. \ref{Fig-9} and Fig. \ref{Fig-10} depict the evolution of two numerical solitons computed by LM-CN at different instants. The contraction and expansion processes of the circular soliton are accurately captured, with no displacement observed at the soliton center. The collision between two circular solitons clearly preserves symmetry with respect to the lines $x=-10$ and $y=-7$, as evident from the contour plots. These numerical results are in precise agreement with prior studies \cite{bo-22-EIEQ-MS,cai-19-SG-NB-JCP,jiang-19-SG-IEQ-JSC}. Furthermore, the proposed Lagrange multiplier schemes consistently preserve the original energy in Fig. \ref{Fig-11}.

\section{Conclusions}\label{sec:conclusions}
In this paper, we present a novel framework for constructing linearly implicit energy-preserving schemes of arbitrary order for Hamiltonian PDEs. The methodology integrates the newly developed Lagrange multiplier approach with symplectic Runge-Kutta methods and a prediction-correction strategy, enabling exact preservation of the original energy in both continuous and discrete settings. The resulting numerical schemes only require solving linear equations with constant coefficients. When combined with the Fourier pseudo-spectral spatial discretization, their computational efficiency is further improved via fast Fourier transforms. We provide rigorous proofs of the energy conservation and numerical accuracy for all schemes. Finally, numerical experiments on three representative Hamiltonian PDEs demonstrate the superior behaviors of the proposed schemes.

\bibliographystyle{siamplain}
\bibliography{references}
\end{document}